\documentclass[11pt]{article}
\usepackage[utf8x]{inputenc}
\usepackage[T1]{fontenc}
\usepackage{amssymb,amsmath,mdwlist,url,verbatim}
\usepackage{amsthm}
\usepackage{amscd}
\usepackage{color}
\usepackage{lastpage}
    \usepackage[pdftex]{graphicx}
    \DeclareGraphicsExtensions{.jpg, .png, .pdf}
\usepackage{marvosym}
\usepackage{multicol}
\usepackage{enumerate}

\usepackage{tikz}
\usetikzlibrary{arrows}
\usetikzlibrary{shapes.misc}
\usetikzlibrary{positioning}
\usetikzlibrary{calc}
\usepackage{ifthen}
\usepackage{hyperref}
\usepackage{cleveref}
\usepackage{lipsum}

\usepackage{tikz,tkz-tab} 

\newcommand\N{{\mathbb N}}

\newcommand\R{{\mathbb R}}

\newcommand\eps{{\varepsilon}}
\newcommand\fP{{\mathbb P}}

\newcommand\LpoI{{\Lambda_{p,I}^0}}
\newcommand\LpoII{{\Lambda_{p,II}^0}}

\newcommand{\pn}{\par\noindent}

\newtheoremstyle{question}	
     {}						
     {1\baselineskip}			
     {\normalfont}				
     {}						
     {\bfseries\sffamily}		
     {.}						
     {\newline}				
     {}						

\newtheorem*{theo*}{Th\'eor\`eme}
\newtheorem*{defin*}{D\'efinition}

\newcommand{\indep}{\rotatebox[origin=c]{90}{$\models$}}

 \definecolor{gris25}{gray}{0.75} 

\newtheorem{thm}{Theorem}[section]

\newtheorem{lem}{Lemma}[section]
\newtheorem{rem}{Remark}
\newtheorem{defn}{Definition}[section]

\newtheorem{prop}{Proposition}[section]

\newcommand{\E}{\mathbb{E}}      

\newcommand{\al}{\alpha} 
\newcommand{\be}{\beta} 
\newcommand{\ep}{\epsilon_n} 
\newcommand{\Lpo}{\Lambda_p^0}
\newcommand{\Lpe}{\Lambda_p^{K}}

\newcommand{\ninf}{\underset{n\rightarrow +\infty}{\longrightarrow}}

\newcommand{\tcoh}{L_{n,\tau}}

\newcommand{\X}{\mathbb{X}}
\newcommand{\fPo}{\mathbb{P}_0}
\newcommand{\fPk}{\mathbb{P}_K}
\newcommand{\liminfini}{\underset{n\rightarrow +\infty}{\lim}}
\newcommand{\fPal}{\mathbb{P}_{\alpha}}
\newcommand{\fPb}{\mathbb{P}_{\beta}}
\newcommand{\Balo}{B_{\alpha}^0}
\newcommand{\BalK}{B_{\alpha}^K}

 \newcommand{\footnoteremember}[2]{
\footnote{#2}
\newcounter{#1}
\setcounter{#1}{\value{footnote}}
}
\newcommand{\footnoterecall}[1]{
\footnotemark[\value{#1}]
}

\begin{document}

\author{M. Boucher \footnoteremember{IDP}{Institut Denis Poisson, Universit\'e d'Orl\'eans, Universit\'e de Tours, CNRS, Route de Chartres, B.P. 6759, 45067, Orl\'eans cedex 2, France}, D. Chauveau\footnoterecall{IDP}, M. Zani \footnoterecall{IDP}\footnote{Corresponding author: marguerite.zani@univ-orleans.fr}
}
\date{}
\title{Coherence of high-dimensional random matrices in a Gaussian case : application of the Chen-Stein method }
\maketitle
\underline{Abstract :}
\newline
\footnotesize
This paper studies the $\tau$-coherence of a $(n \times p)$-observation matrix in a Gaussian framework.
The $\tau$-coherence is defined as the largest magnitude outside a diagonal bandwith of size $\tau$ of the empirical correlation  coefficients associated to our observations. 
Using the Chen-Stein method we derive the limiting law of the normalized coherence and show the convergence towards a Gumbel distribution. We generalize here the results of Cai and Jiang \cite{CaiJiang}. We assume that the covariance matrix of the model is bandwise. Moreover, we provide numerical considerations highlighting issues from the high dimension hypotheses. We numerically illustrate the asymptotic behaviour of the coherence with Monte-Carlo experiment using a HPC splitting strategy for high dimensional correlation matrices.
\normalsize

\section{Introduction}
Random matrix theory has known a huge amount of breakthroughs for these last twenty years. Developments have been made in theoretical fields as well as in various applied domains. Among these applications, one can cite high-energy physics (e.g. \cite{Forrester} on log--gases), electronic engineering (signal and imaging, see \cite{Donoho, CandesTao,CandesRombergTao1,CandesRombergTao2} ), statistics (see \cite{Johnstone1,Johnstone2,butucea2013sharp}).
Earlier works on random matrices were focused on spectral analysis of eigenvalues and eigenvectors (see \cite{wigner} or \cite{Meh,BaiSilver}, see also \cite{BordeChafai} and references therein). For a reference on random matrices theory, see \cite{BaiSilver,Meh,AndGuioZeit}.
\par
In statistics more particularly, random matrices  are useful for inference in a high dimensional framework. One can think about high dimensional regression, hypothesis testing for high dimension parameters, inference for large covariance matrices.  See e.g. \cite{BaiSarana,CandesTao1,BaiJiangYao,Caietal1,Caietal3,BickelTsybakov}. In these contexts, the dimension $p$ is much bigger than the sample size $n$.
\par
We will be focusing here on the covariance structure of a certain type of random matrices. More precisely, we will be examining the coherence of random matrices with bandwise covariance of size $\tau>1$. Let us define the model.
\par
Let $\left(X^1, X^2, \dots, X^p\right)$ be a $p$ -- dimensional Gaussian random vector with mean\\ $\mu=^t(\mu^1,\cdots, \mu^p)$ in  $\R^p$ and covariance matrix $\Sigma=(\sigma_{kj})_{kj} \in \R^{p}\times \R^{p}$. 
For $n\in \N^*$, we consider a random sample $\left(X_i^1, X_i^2, \dots, X_i^p\right)_{i=1, \dots, n}$ issued from $\left(X^1, X^2, \dots, X^p\right)$, arranged in a 
$(n,p)$-matrix $\mathbb{X}$. From a statistical point of view, it means that each row can be seen as an individual and each column as a character. We will write $\mathbb{X}^k$ the $k^{th}$ column of $\mathbb{X}$. 
We are interested in the correlation terms of $\mathbb{X}$ in models where $p$ is much larger than $n$. The classical empirical Pearson's correlation coefficient is defined by : 

\begin{equation} \rho_{kj}=\frac{\sum\limits_{i=1}^{n}\left(X_i^k-\overline{\X^k}\right)\left(X_i^j-\overline{\X^j}\right)}{\sqrt{\sum\limits_{i=1}^{n}\left(X_i^k-\overline{\X^k}\right)^2}\sqrt{\sum\limits_{i=1}^{n}\left(X_i^j-\overline{\X^j}\right)^2}}=\frac{<\X^k-\overline{\X^k}\mathbf{1}_n,\X^j-\overline{\X^j}\mathbf{1}_n>}{\parallel \X^k-\overline{\X^k}\mathbf{1}_n \parallel . \parallel \X^j-\overline{\X^j}\mathbf{1}_n \parallel}\,,
\end{equation}
where $\mathbf{1}_n$ is the identical $1$--vector in $\mathbb R^n$:
\[\mathbf{1}_n=^t(1, 1, \dots, 1)\in \R^n\]
and $\parallel x \parallel$ stands for the the Euclidian norm of the vector $x$ and $\overline{\X^k}$ is the empirical mean of the $k^{th}$ column $\X^k$:
\[\overline{\X^k}=\frac{1}{n}\sum\limits_{i=1}^{n}X_i^k\, . \]
Equivalently if the mean $\mu= {}^t\left(\mu^1, \dots, \mu^p\right)$ is known,

\begin{equation} \tilde\rho_{kj}=\frac{\sum\limits_{i=1}^{n}\left(X_i^k-\mu^k\right)\left(X_i^j-\mu^j\right)}{\sqrt{\sum\limits_{i=1}^{n}\left(X_i^k-\mu^k\right)^2}\sqrt{\sum\limits_{i=1}^{n}\left(X_i^j-\mu^j\right)^2}}=\frac{<\X^k-\mu^k\mathbf{1}_n,\X^j-\mu^j\mathbf{1}_n>}{\parallel \X^k-\mu^k\mathbf{1}_n \parallel . \parallel \X^j-\mu^j\mathbf{1}_n \parallel}
\end{equation}

The empirical correlation coefficients $\rho_{kj}$ (resp. $\tilde\rho_{kj}$) are arranged in a $(p,p)$-matrix $R_n$ (resp. $\tilde R_n$) which is the empirical correlation matrix of $\X$. 
\begin{defn}
	With the notations above, we can define the largest magnitude of the off--diagonal terms of $R_n$ and $\tilde R_n$:
	\begin{equation}\label{defcoherence}
	L_n=\max_{1\leq k<j\leq p}|\rho_{kj}|\,,\quad\tilde L_n=\max_{1\leq k<j\leq p}|\tilde\rho_{kj}|
	\end{equation}
	The quantity $\tilde L_n$ is defined as the coherence of the matrix $\mathbb {X}_n$. With a slight abuse of terminology, we will call both $L_n$ and $\tilde L_n$ coherence of $\mathbb {X}_n$.
\end{defn}
The notion of coherence has first appeared in signal theory as an indicator of the sparsity of a matrix. More precisely, it is involved in the so-called Mutual Incoherence Property (MIP), which can be explained as follows: a mesurement $(n,p)$ matrix $X$ is used to recover a $k$--sparse signal $\beta$ via linear mesurements $y=X\beta$ using a recovery algorithm. The condition
\[(2k-1)\tilde L_n<1\]
ensures the exact recovery of $\beta$ when $\beta$ has at most $k$ non zero entries. For details on this approach, see Donoho and Huo \cite{DonohoHuo}, Fuchs \cite{Fuchs}, Cai, Wang and Xu \cite{Caietal2}, and references therein.
\par
Another domain where covariance and correlation matrices are highly used is statistic theory, for example testing
$\Sigma= I$ against $\Sigma\neq I$. This issue has been considered in the case where $n$ and $p$ are of same order (i.e. $n/p\to \gamma\in(0,\infty)$ ) by Johnstone in the Gaussian case \cite{Johnstone1}, and Péché in the sub-Gaussian case \cite{Peche}. The test statistic relies -- according to PCA methods -- on the largest eigenvalue of the empirical covariance matrix $\lambda_{max}(\hat\Sigma_n)$. The asymptotic distribution of this maximum eigenvalue is the Tracy--Widom law. 
\par
However, testing $\Sigma=I$ can seem too restrictive, and one think about independance versus non independence in terms of correlation matrix i.e. testing
$R_n=I$ against $R_n\neq I$. According to previous results, one could think about using $\lambda_{max}(R_n)$. However, even if Tracy--Widom law is conjectured in this case (see \cite{Jiang04cor}, and see also \cite{HeinyMikosch} for a study on the i.i.d. case), one choose to study instead the coherence as a test statistic.
Jiang in \cite{Jiang04} first adressed this problem and showed strong consistency of $L_n$ and limit distribution of $L_n^2$ in the case where $n$ and $p$ are of the same order. Moments assumptions in \cite{Jiang04} and dimension for $p$ were substantially improved by a series of papers: 
Li and Rosalsky \cite{Rosal1}, Zhou \cite{ZHou07}, Liu, Lin and Shao \cite{LiuLinShao}, Li, Liu and Rosalsky \cite{Rosal2}, Li, Qi and Rosalsky \cite{Rosal3}.
In \cite{CaiJiangSpher} the authors consider the limiting distribution of the coherence in a spherical case. See also \cite{CaiZhangdiffcorr} for studies on the differential correlation matrices in high dimensional context.
\par
Lately Cai and Jiang \cite{CaiJiang} (see also the supplement \cite{CaiJiangPlus}) considered "ultra-high dimensions" i.e. $p$ as large as $e^{n^{\beta}}$. In this paper, they also present a variant of the coherence, the so--called $\tau$--coherence
aimed to test whether the covariance $\Sigma$ has a given bandwidth $\tau>1$, where $\tau=1$ would be a special case. We define it below:
\begin{defn}
	For any integer $\tau\geq 1$, we define the $\tau$--coherence as:
	\begin{equation}
	\tcoh = \underset{|k-j|\geqslant \tau}{\max}\left|\rho_{kj}\right|
	\end{equation}
\end{defn}
In \cite{CaiJiang} strong laws and convergence of distributions of $L_{n,\tau}$ are given as well. Recently, Shao and Zhou \cite{shao2014} studied coherence and $\tau$ coherence relaxing the normal hypothesis, improving assumptions on the moments of the entries and on the dimension $p$.

Our purpose is to generalize this model: 
we assume that $\Sigma = \left( \sigma_{kj} \right)_{1\leqslant k,j \leqslant p}$ is defined as follows:

\begin{equation}
\sigma_{kj}=
\left\lbrace
\begin{array}{cccc}
\gamma_{kj}\sigma_k\sigma_j  & \mbox{if} & |k-j|  < \tau\\
\epsilon_n\sigma_k\sigma_j & \mbox{if} & \tau \leq |k-j| \leq \tau + K \\
0 & \mbox{if} & \tau + K < |k-j| 
\end{array}\right..
\end{equation}

So, $\Sigma$ is divided into three parts : a central band of size $\tau\in \N$ , an outside part with null coefficients and a transitional bandwidth of size $K \in \N$. We have, for all $k \in [\![1;p]\!]$, $\sigma_k >0$; for all $k,j \in [\![1;p]\!]$, $\gamma_{kj} \in [-1,1]$ and $\left(\ep\right)_{n\geqslant 1}$ is a sequence of real numbers in $[-1,1]$ such that $\underset{n\rightarrow +\infty}{\lim}\left|\ep\right|=0$. To be more precise, the construction of $\Sigma$ suggests that if we take two $p$-vectors which are close (in term of indexes, for example $X^1$ and $X^2$), they will be correlated. If they are far enough one to another, they will be independent. But, if they are not so close and not so far, their correlation will decrease to zero when $n$ goes to infinity. This generalization of the model of \cite{CaiJiang} seemed to us more suited for real datasets. Later, we will see that both $\tau$ and $K$ may depend of $n$ and may go to infinity under sufficient hypotheses. Without loss of generality, we can assume that $\mu = 0_{\R^p}$ and for all $k \in [\![1,p]\!]$, $\sigma_k =1$.\\

All previously cited studies are highly related to the Chen-Stein method which we also use here. It relies on a Poisson approximation of weakly dependent events. For references on this method, see \cite{MR972770}, \cite{peccati} and references therein.\\

This paper is organized as follows: \cref{results} presents the main results, further on \cref{simu} gives some simulation results for our model, \cref{proofs} is devoted to the proof of the main result whereas \cref{tech} gather technical results and proofs of technical lemmas.
\par
The usual notations $u_n=o(v_n)$ and $u_n=O(v_n)$ stand for $u_n$ negligible with respect to $v_n$ and $u_n$ of the same order of $v_n$ respectively and asymptoticaly when $n\to\infty$.

\section{Main result}\label{results}
We focus on correlations not too big, i.e. not too close to $1$ or $-1$. Hence we define the following set:
\begin{defn}
	For any $\delta \in ]0,1[ $ we define by
	\[ \Gamma_{p,\delta}=\{ k \in [\![1;p]\!] \text{ : } |r_{kj}|> 1-\delta \text{ for some } j \in [\![1;p]\!] \text{ and } k \neq j\},\]
	where $(r_{kj})_{p\times p}$ is the correlation matrix issued from $\Sigma=(\sigma_{kj})_{p\times p}$.
\end{defn}

The main result of the paper is the following Theorem:
\begin{thm}\label{main_result}
	Let $n$ be an integer, $p=p_n$ a sequence such that $p_n \underset{n\rightarrow +\infty}{\longrightarrow} +\infty$. Let $(\epsilon_n)_{n\in \N^*}$ be a sequence of real number in $]-1,1[$. Let us assume the following conditions :
	\begin{enumerate}
		\item[Hyp 1 :] $\log(p_n)=o(n^{\frac{1}{3}})$ as $n\rightarrow +\infty$ \label{hyp1}
		\item[Hyp 2 :]\label{hyp2} $\tau=\tau(n)=o(p_n^t)$ as $n\rightarrow +\infty$ for any $t>0$.
		\item[Hyp 3 :] $\exists \delta \in ]0,1[$ such that $|\Gamma_{p,\delta}|=o(p_n)$ where $|\cdot|$ denotes the cardinality of the set.\label{hyp3}
		\item[Hyp 4 :] $\epsilon_n \sim \gamma \sqrt{\frac{\log(p_n)}{n}}$ as $n\rightarrow +\infty$ and $\gamma \in ]-2+\sqrt 2,2-\sqrt 2[$ \label{hyp4}
		\item[Hyp 5 :] $K=K(n)=O\left(p_n^{\nu}\right)$ where $\nu \in ]0, c(\gamma,\delta)[$ \label{hyp5}
	\end{enumerate}
	and $c(\gamma,\delta)=min\left(\frac{1}{3}(\frac{1}{2}\gamma^2-2|\gamma|+1),\frac{\delta^2(2-\delta)^2}{36}\right)$.\\
	Under these conditions, we can show that :
	\begin{equation}\label{mainlimit}
	nL_{n,\tau}^2-4\log(p_n)+\log(\log(p_n))\underset{n\rightarrow +\infty}{\overset{\mathcal{L}}{\longrightarrow}}Z
	\end{equation}
	where $Z$ has the cdf $F(y)=e^{-\frac{1}{\sqrt{8\pi}}e^{-\frac{y}{2}}}$ for all $y \in \R$.
\end{thm}

We can observe that it is the same distribution as in \cite{CaiJiang}. The band in the covariance matrix which contains terms $\eps_n$ is a smooth transition between the central band and the external part of the matrix with null terms. This model is an illustration of vanishing dependence when the components are too far one from each other. In \cite{CaiJiang}, the central band of the matrix is as large as $\tau$ with the condition of theorem \ref{hyp3}. If we suppose $K = cst < +\infty$, we boil down to the same model. Indeed, we have $\tilde{\tau}=\tau + K$ wich is the new width of the non--null bandwidth, and $\tilde{\tau}$ is still such that  :
\begin{equation}
\forall t>0 \text{,   } \tilde{\tau}=o(p^t) \text{  as  } n \rightarrow +\infty.
\end{equation}

\section{Numerical aspects}\label{simu}

In this section, we provide some simulated examples to illustrate the behavior of our asymptotic result in practical simulations ($n$ and $p$ both large but finite). For this, we use the R Statistical Software \cite{RSoft}. A difficulty comes from the fact that in our context, we have to compute correlations of large matrices. We need Gaussian observation matrices of size $n \times p$ with $\log\left(p\right) = o\left(n^{\frac{1}{3}}\right)$. It means that for a large $n$, for example $n = 4000$, we will have $p \approx 45000$ taking $p = \left[\exp\left( n^{\frac{1}{3.5}}\right)\right]   $ in our simulations (where $\left[x\right]$ is the integer part of $x$). For each $(n \times p)$-observation matrix, we have to compute the $(p \times p)$-correlation matrix to compute the $\tau$-coherence. For the range of $p$ that we consider, we can observe the evolution of the size of the $(p \times p)$-matrix in Gb according to $n$ in Figure \ref{fig:tailleGb}. For example, with $n = 4000$ and $p = 44112$, we have, for correlation stored in double, a $14.5 Gb$ $(p \times p)$-matrix which is very large for a common computer. We must find a way to compute the $\tau$-coherence without loading the entire $(p \times p)$-correlation matrix in the computer memory (RAM).\\
\vspace{-0.83cm}
\begin{figure}[h]
	\begin{center}
		\includegraphics[width=0.7\linewidth]{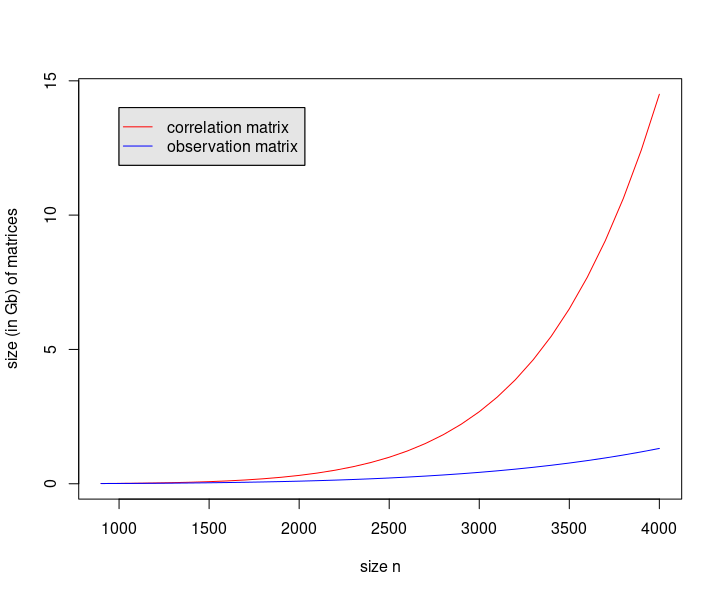}
	\end{center}
	\caption{\small Size of $(p \times p)$-correlation matrix and $(n \times p)$-observation matrix according to $n$ with $p=\left[\exp(n^{1/3.5})\right]$ for real numbers stored as double
		precision numbers.}
	\label{fig:tailleGb}
\end{figure}

The idea is to generate the $(n \times p)$-observation matrix by packets of columns. Each packet will have a size $(n \times Tb)$ where $Tb$ is choosen by the user. With these packets of columns, we compute all correlation  blocks of size $(Tb \times Tb)$ between each pair of packets of columns. In that way, we must choose a size $Tb$ in order to have two blocks fitting simultaneously in the computer memory. Then, we can compute the $\tau$-coherence by taking the largest coefficient in absolute value in our block paying attention to wether the block corresponds to the central band (with bandwith $\tau$) or not.\\ 

Using this strategy, we can generate correlation matrices even if $p$ is very large, so that we are able to study the limiting distribution of the $\tau$-coherence. In that way, to illustrate our theorem, we consider the following parameters :
$$ p = \left[exp\left(n^{1/3.5}\right)\right] \text{   },\text{   } \tau = 5*\left[\log(p)\right] \text{   },\text{   } K = 10*\left[n^{1/10}\log(p)\right] \text{   },\text{   } \eps_n = 0.1*\sqrt{\frac{\log(p)}{n}}$$

Our purpose here is to simulate a sample of $\tau$-coherence by a Monte-Carlo procedure in order to compare its empirical distribution with the asymptotic one. We thus run $R = 200$ replications of the following procedure, simulating $R$ times the matrices of observations and computing the correlations per blocks. For each replications, we generate an observation matrix $\mathbf{X}$ of size $(n \times p)$ using the following numerical scheme : 

\begin{equation}\forall i \in [\![1,n]\!], \forall j \in [\![1,p]\!], X_i^j= \sum\limits_{k=j}^{j+K-1}\eps_nY_i^k + \sum\limits_{k=j+K}^{j+K+2\tau}r_kY_i^k+\sum\limits_{j+K+2\tau+1}^{j+2\tau+2K}\eps_nY_i^k
\end{equation}
where all coefficients $\left(r_k\right)_{1 \leqslant k \leqslant 1+2\tau}$ are real numbers in $[-1,1]$ (we take\\ $r_1, \dots r_{1+2\tau} \overset{i.i.d}{\sim} \mathcal{U}_{[-1,1]}$ in the simulation), $X_i^j$ is the coefficient of $\mathbf{X}$ on the $i^{th}$ line and the $j^{th}$ column and all random variable $Y_i^k \overset{i.i.d}{\sim} \mathcal{N}\left(0,1\right)$ arranged in a $\left(n \times (p+2\tau+2K) \right)$-matrix $\mathbf{Y}$. We highlight the fact that $\mathbf{Y}$ is quite larger than $\mathbf{X}$. This numerical scheme is inspired by time series model.
\medskip\pn

\begin{figure}[h]
	\begin{center}
		\includegraphics[width=0.65\linewidth]{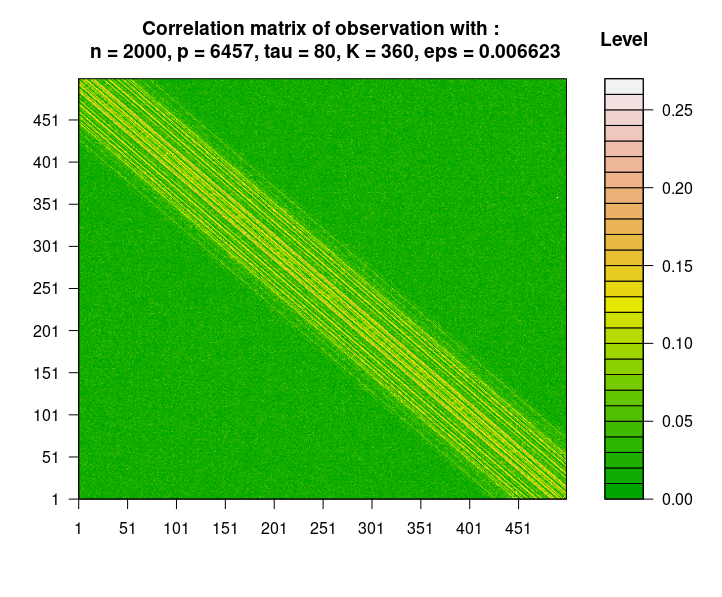}
	\end{center}
	\caption{\small Level plot of the correlation's structure with observations : zoom on the square $1:500$}
	\label{fig:modelecor1000}
\end{figure}

We can observe that we generate data following our model and obtain an observation matrix associated to a correlation matrix with a band structure in Figure \ref{fig:modelecor1000}. We recognize a central band with non-null coefficients. In fact, we also notice that the transition band with $\eps_n$' coefficients is not really recognizable but this is due to the fact that those coefficients are decreasing fastly to $0$ when $n$ goes to infinity (for instance, here, we have $\eps_n \approx 0.007$ not different from $0$ in the color scale).\\

With this observation matrix, we can use our procedure to compute the $\tau$-coherence. After running $R$ replications, we obtain a sample of $\tau$-coherence. In Figure \ref{fig:histregime14000}, we see that for $n$ large enough, the sample distribution seems to approximate the limiting one. 

\begin{figure}[h]
	\centering
	\includegraphics[width=0.8\linewidth]{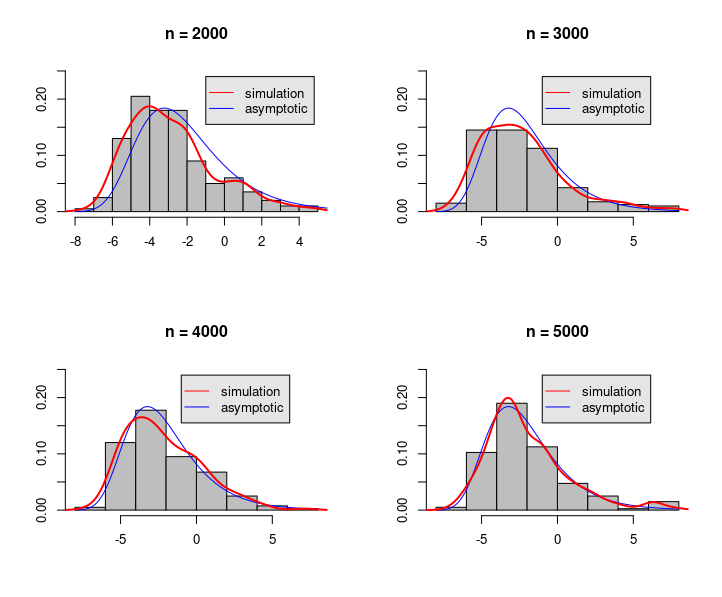}
	\caption{\small Histograms and Kernel
		density estimates for $n = 2000,3000,4000,5000$ and for $R = 200$ replications}
	\label{fig:histregime14000}
\end{figure}

Precisely, we compare the estimated density of the sample (in red) with the asymptotic density (in blue) which is defined by $f(x) = \frac{1}{2\sqrt{8\pi}}\exp\left(-\frac{1}{2}y -\frac{1}{\sqrt{8\pi}}\exp\left(-\frac{1}{2}y\right)\right)$ for all $x \in \mathbb{R}$. Also, in order to observe the convergence, we study numerically the distance between the sample and asymptotic distribution. We use the Kolmogorov, $\mathbb{L}^2$ and the Total Variation norms. We remind, respectively, the definition of these norms :
\begin{eqnarray}
d_{KS}(\hat{f},f)=\underset{x \in \mathbb{R}}{\sup}\left|F_n(x)-F(x)\right|,
\end{eqnarray}
\begin{eqnarray}
d_2(\hat{f},f)=\int \left|\hat{f}(x)-f(x)\right|^2dx,
\end{eqnarray}
\begin{eqnarray}
d_{TV}(\hat{f},f)=\frac{1}{2}\int \left|\hat{f}(x)-f(x)\right|dx.
\end{eqnarray}
We observe, in the results displayed in Figure \ref{fig:article_norm}, that the difference between both distributions decreases to $0$ when $n$ is increasing.\\

\begin{figure}[h]
	\centering
	\includegraphics[width=0.7\linewidth]{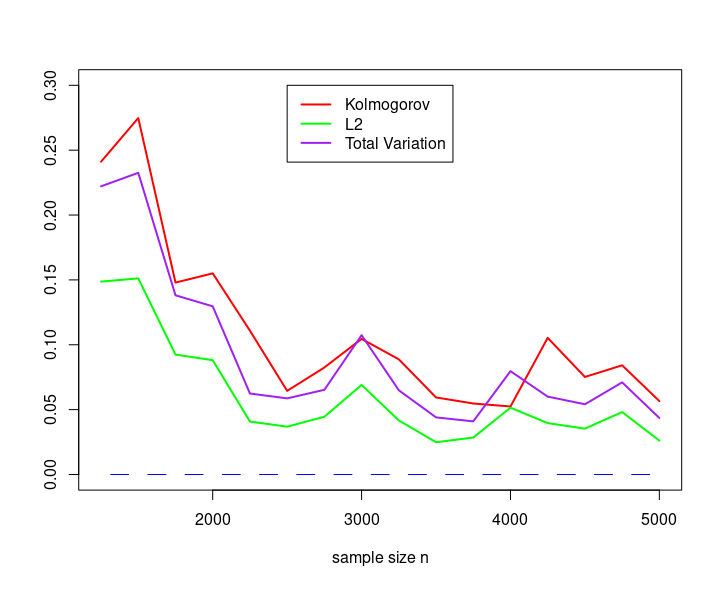}
	\caption{\small Evolution of Kolmogorov, $\mathbb{L}^2$ and Total Variation norm between simulated and asymptotic behavior}
	\label{fig:article_norm}
\end{figure}
These results provide numerical evidence that our limiting distribution is adequate. We also higlight the fact that the procedure we proposed here allows to compute $\tau$-coherence corresponding to any large matrix $X$ arising in actual (big) data experiments. However, this procedure is not very efficient if it is done with a classical programming. For example, computing only one replication for $n = 4000$ (and so $p= 44112$), requires about $90$ min to obtain the value of one $\tau$-coherence. In order to obtain more usable (i.e fast) codes in perspective of real-size applications, we are currently exploring HPC strategies to compute correlation blocks using GPGPU computation. We are very confident into the use of GPU to reduce simulation's time.


\section{Proof of the main result}\label{proofs}
\par
In this section, we describe the proof of our main result. First, we would like to higlight the fact that, as we said, we apply the Chen-Stein method. But, we do not apply it directly to the $\tau$-coherence. It is more efficient to use the Chen-Stein method to a new easier to handle random variable.
\medskip\pn
First of all, we introduce many notation which will be used along this paper.

\subsection{Notations}\label{nota}

\begin{itemize}
	
	\item $I=\{(k,j)\in [\![1,p]\!]^2 : 1 \leqslant k < j \leqslant p \}$
	\item $I_{\tau}=\{(k,j) \in I : |k-j| < \tau \}$
	\item $I_{K}=\{(k,j) \in I : \tau \leq |k-j| \leq \tau + K \}$
	\item $I_0=\{(k,j)\in I : |k-j| > \tau + K \}$
	\item $E_{\delta}=\{(k,j)\in I: k \in \Gamma_{p,\delta} \text{ or } j \in \Gamma_{p,\delta} \}$
	\item $\Lambda_p^{\tau}=\{(k,j) \in I : |k-j|<\tau \text{ and } \underset{1\leqslant k\neq q,j\neq q \leqslant p}{\max}\left(|r_{kq}|,|r_{jq}|\right)\leqslant 1-\delta \}$
	\item $\Lambda_p^{K}=\{(k,j) \in I : \tau \leq |k-j| \leq \tau+K \text{ and } \underset{1\leqslant k\neq q,j\neq q \leqslant p}{\max}\left(|r_{kq}|,|r_{jq}|\right)\leqslant 1-\delta \}$
	\item $\Lambda_p^{0}=\{(k,j) \in I : |k-j|>\tau + K \text{ and } \underset{1\leqslant k\neq q,j\neq q \leqslant p}{\max}\left(|r_{kq}|,|r_{jq}|\right)\leqslant 1-\delta \}$
\end{itemize}

With these different sets, we can write three different partitions of the set $I$ :
\begin{enumerate}
	\item $I = I_{\tau} \cup I_{K} \cup I_{0}$
	\item $I = E_{\delta} \cup \Lambda_p^{\tau} \cup \Lambda_p^{K} \cup \Lambda_p^{0}$
	\item $I_0 \cup I_{K} = \Lpe \cup \Lpo \cup \left[E_{\delta}\cap \overline{I_{\tau}}\right]$
\end{enumerate}
The following Lemma gives the sizes of these three sets:
\begin{lem}\label{card0}
	With the previous notations
	\begin{equation}\label{cardItau}
	|I_{\tau}|=(\tau -1)\left(\frac{2p-\tau}{2}\right),                                            
	\end{equation}
	\begin{equation}\label{cardIepsi}
	|I_{K}|=(K+1)\left(\frac{2p-K-2\tau}{2}\right),
	\end{equation}
	\begin{equation}\label{cardI0}
	|I_0|=\frac{(p-\tau-K-1)(p-\tau-K)}{2}.
	\end{equation}
\end{lem}

\subsection{Auxiliary variables}
Now we introduce an auxiliary random variable which will be more convenient to handle in the Chen--Stein method.
Let
\begin{equation}\label{vntau}
{V}_{n,\tau}=\underset{1\leqslant k < j \leqslant p,|k-j|\geqslant \tau}{\max}|{}^tX^kX^j|=\underset{\al=(k,j) \in I_0\cup I_{K}}{\max}\left|{}^tX^kX^j\right|
\end{equation}
In the sequel, we will use notation $\alpha=(k,j)$ to denote index into different sets.


\begin{prop}\label{convLnvn}
	Under the assumptions of Theorem theorem \ref{main_result}, we have
	\begin{equation} \frac{n^2L_{n,\tau}^2-V_{n,\tau}^2}{n}\underset{n\rightarrow +\infty}{\overset{\mathbb{P}}{\longrightarrow}}0.
	\end{equation}
\end{prop}
\pn
The proof of this Proposition is postponed to \cref{proofs}.
\medskip\pn
Hence to study the asymptotic behaviour of $L_{n,\tau}$, it is enough to study the limiting distribution of ${V}_{n,\tau}$. To do so, we use another slightly different random variable defined by:
\begin{equation}
{V}_{n,\tau}'=\underset{\al \in \Lpo \cup \Lpe}{\max}\left(Z_{\al}\right)
\end{equation}
where the index $\al = (k,j)$ and $Z_{\al}=Z_{kj}=\left|{}^tX^kX^j\right|$. The two variables ${V}_{n,\tau}$ and ${V}_{n,\tau}'$ are linked by the following inequalities:

\begin{prop}\label{equiVtau}
	Let 
	\begin{equation}\label{an}
	a_n(y)=\sqrt{4n\log(p_n)-n\log\log(p_n)+ny}\mbox{ with }y\in \R.
	\end{equation}
	We have :
	\begin{equation}
	\fP\left({V}_{n,\tau}'>a_n(y)\right)  \leqslant  \fP\left({V}_{n,\tau}>a_n(y)\right) \leqslant  \fP\left({V}_{n,\tau}'>a_n(y)\right)+ o(1)
	\end{equation}
\end{prop}

\begin{proof}
	(For seek of simplicity, we will denote $a_n(y)$ by $a_n$ in the sequel).
	\\
	To proove this result, we need the two following technical results whose proofs are postponed to \cref{proofs}.

	\begin{lem}\label{probI0}
		Let $a_n$ be as in formula \eqref{an}. Then, 
		\begin{equation} 
		\fPo := \fP\left(\left| {}^tX^1X^{\tau+ K +2}\right|>a_n\right)=\frac{1}{\sqrt{2\pi}}e^{-\frac{y}{2}}\frac{1}{p_n^2}\left(1+o(1)\right)=O_{n\rightarrow +\infty}\left(\frac{1}{p_n^2}\right)
		\end{equation} 
	\end{lem}
	
	\begin{lem}\label{probIk}
		Let $a_n$ be as in formula \eqref{an} and let us define $c_{\gamma} := \frac{1}{2}\gamma^2-2|\gamma|+2$ with $\gamma$ defined in \cref{hyp3}. Then, for any $d \in [0;c_{\gamma}[$ and $n \rightarrow +\infty$ :
		\begin{equation} 
		\fPk:= \fP\left(\left| {}^tX^1X^{\tau+1}\right|>a_n\right)=o\left(p_n^{-d}\right)
		\end{equation} 
	\end{lem}

	According to the partition $I_0 \cup I_K = \Lpo \cup \Lpe \cup \left(E_{\delta}\cap \overline{I_{\tau}}\right)$,
	\begin{eqnarray*}
		\fP\left(V_{n,\tau}>a_n\right)  & = & \fP\left(\underset{\al=(k,j) \in I_0\cup I_{K}}{\max}\left|{}^tX^kX^j\right|>a_n\right)\\		
		& \leqslant & \fP\left(V_{n,\tau}'>a_n\right)+\fP\left(\underset{\al=(k,j) \in E_{\delta}\cap \overline{I_{\tau}}}{\max}\left|{}^tX^kX^j\right|>a_n\right)\\
		& \leqslant & \fP\left(V_{n,\tau}'>a_n\right)+\sum\limits_{\al=(k,j) \in E_{\delta}\cap \overline{I_{\tau}}}\fP\left(\left|{}^tX^kX^j\right|>a_n\right)\\
		& \leqslant & \fP\left(V_{n,\tau}'>a_n\right)+\sum\limits_{\al\in \left[E_{\delta}\cap \overline{I_{\tau}}\right]\cap I_{K}}\fP\left(Z_{\al}>a_n\right)+\sum\limits_{\al\in \left[E_{\delta}\cap \overline{I_{\tau}}\right]\cap I_{0}}\fP\left(Z_{\al}>a_n\right).
	\end{eqnarray*}	
	All variables $Z_{\alpha}$ having same distributions in the different sets above, we have		
	\begin{eqnarray*}
		\fP\left(V_{n,\tau}>a_n\right) 		
		& {\leqslant} & \fP\left(V_{n,\tau}'>a_n\right)+\left| \left[E_{\delta}\cap \overline{I_{\tau}}\right]\cap I_{K}\right|\fP\left(Z_{1,\tau+1}>a_n\right)+\left|\left[E_{\delta}\cap \overline{I_{\tau}}\right]\cap I_{0}\right|\fP\left(Z_{1,\tau+K+2}>a_n\right)\\
		& \leqslant & \fP\left(V_{n,\tau}'>a_n\right)+\left|I_K\right|\fPk+\left|E_{\delta}\right|\fPo.\\
	\end{eqnarray*}
	We can use the following straightforward result :
	\begin{lem}
		\begin{equation*}
		|E_{\delta}| \leqslant 2p_n|\Gamma_{p,\delta}|
		\end{equation*}
	\end{lem}
	Hence from assumption 3 of theorem \ref{main_result}, we have : 
	\begin{equation}\label{cardEdelta0}
	\left|E_{\delta}\right| = o(p_n^2)
	\end{equation}
	Now, we need to prove that $\left|I_K\right|\fPk+\left|E_{\delta}\right|\fPo \ninf 0$. First, from lemma \ref{probI0} and \eqref{cardEdelta0},
	\begin{eqnarray*}
		\left|E_{\delta}\right|\fP\left(Z_{1,\tau+K+2}>a_n\right) &{\underset{n\rightarrow +\infty}{\sim}} & \left|E_{\delta}\right|\frac{1}{\sqrt{2\pi}}e^{-\frac{y}{2}}\frac{1}{p_n^2}\\
		&{\underset{n\rightarrow +\infty}{=}} & o(p_n^2)\frac{1}{\sqrt{2\pi}}e^{-\frac{y}{2}}\frac{1}{p_n^2}\\
		& = & \frac{1}{\sqrt{2\pi}}e^{-\frac{y}{2}}o(1) \ninf 0.\\
	\end{eqnarray*}
	Secondly, using lemma \ref{card0} (more precisely \cref{cardIepsi}) and lemma \ref{probIk}  we have :
	\begin{eqnarray}
	\left|I_K\right|\fPk \ninf 0 & \Leftrightarrow & \nu < c_{\gamma} - 1\\
	\end{eqnarray}
	and this is fullfilled from assumptions on theorem \ref{hyp5}. Finally, we obtain : 
	\begin{equation}
	\left|I_K\right|\fPk+\left|E_{\delta}\right|\fPo \ninf 0.
	\end{equation}
	Then,
	\begin{equation}
	\fP\left(V_{n,\tau}>a_n\right)  \leqslant  \fP\left(V_{n,\tau}'>a_n\right)+ o(1)
	\end{equation}
	Also, it is easy to see that :
	\begin{equation}
	\fP\left(V_{n,\tau}'>a_n\right)  \leqslant  \fP\left(V_{n,\tau}>a_n\right).
	\end{equation}
\end{proof}
\begin{rem}\label{rem1}
	The main constraint  so far is
	$p_nK\fPk\to 0\mbox{ when }n\to \infty$
	which leads to
	\[\nu<\frac{1}{2}\gamma^2-2|\gamma|+1\,.\]
Moreover, it also implies the following condition:
	\[\gamma\in[-2,2]\mbox{ is such that } \frac{1}{2}\gamma^2-2|\gamma|+1>0\iff \gamma\in]2+\sqrt 2;2-\sqrt 2[\]
\end{rem}
\subsection{Chen--Stein method for ${V}_{n,\tau}'$}\label{chenStenVntauprime}

We focus now on the asymptotic behaviour of ${V}_{n,\tau}'$. For that purpose, we apply the Chen-Stein method. We remind this result, which can be found in \cite{MR972770}, in the following lemma: 

\begin{lem}{The Chen-Stein Method}\\
	Let $\cal I$ be a set of indices. Let $\alpha \in{\cal  I}$ and $B_{\alpha}$ a set of subset of $\cal I$ (i.e. for all $\alpha$, $B_{\alpha} \subset {\cal I}$). Let $\eta_{\alpha}$ be random variables. For a given $t \in \R$, we define $\lambda := \sum\limits_{\alpha \in {\cal I}}\fP\left(\eta_{\alpha}>t\right)$. Then, 
	\begin{equation} \left|\fP\left(\underset{\alpha\in {\cal I}}{\max}\left(\eta_{\alpha}\right)\leqslant t\right)-e^{-\lambda}\right|\leqslant \min\left(1,\frac{1}{\lambda}\right). \left(b_1+b_2+b_3\right),
	\end{equation}
	where 
	\begin{itemize}
		\item $b_1=\sum\limits_{\alpha\in {\cal I}}\sum\limits_{\beta\in B_{\alpha}}\fP(\eta_{\alpha}>t)\fP\left(\eta_{\beta}>t\right)$
		\item $b_2=\sum\limits_{\alpha\in {\cal I}}\sum\limits_{\alpha \neq\beta\in B_{\alpha}}\fP\left(\eta_{\alpha}>t,\eta_{\beta}>t\right)$
		\item $b_3=\sum\limits_{\alpha\in {\cal I}}\E\left[\left|\E[\mathbf{1}_{\eta_{\alpha}>t}|\sigma(\eta_{\beta},\beta \in {\cal I}\backslash B_{\alpha})]-\E[\mathbf{1}_{\eta_{\alpha}>t}]\right|\right]$.
	\end{itemize}
\end{lem}

As we said, this method is an approximation of weekly dependent events by a Poisson law which is represented by the quantity $e^{-\lambda}$ ( corresponding to $\fP\left(Z=0\right)$, $Z$ having a Poisson law $\mathcal{P}\left(\lambda\right)$). We need to find weekly dependent events to have $b_1$, $b_2$ and $b_3$ small (even null or asymptotically null).\\
\\
In our case, notations are :
\begin{itemize}
	\item $\Lambda=\Lpo \cup \Lpe$.
	\item $\al = (k,j) \in \Lambda$.
	\item $B_{\al}=B_{kj}=\{(u,v)\in \Lambda : |k-u|< \tau +K, |j-v| < \tau +K \text{ and } (k,j)\neq (u,v)\}$.
	\item $\eta_{\alpha}=Z_{\alpha}=Z_{kj}=\left|{}^tX^kX^j\right|=\left|\sum\limits_{i=1}^{n}X_i^kX_i^j\right|$.
	\item $\lambda_n=\sum\limits_{\alpha \in \Lambda}\fP(Z_{\al}>a_n)$.
	\item $b_{1,n}=\sum\limits_{\alpha\in \Lambda}\sum\limits_{\beta\in B_{\al}}\fP(Z_{\al}>a_n)\fP(Z_{\al}>a_n).$
	\item $b_{2,n}=\sum\limits_{\al \in \Lambda}\sum\limits_{\al \neq \be \in B_{\al}}\fP(Z_{\al}>a_n,Z_{\be}>a_n)$.
	\item $b_{3,n}=\sum\limits_{\alpha \in \Lambda}\E\left[\left|\E\left[\mathbf{1}_{Z_{\alpha}>a_n} | \sigma\left(Z_{\beta}, \beta \in \Lambda\backslash B_{\alpha}\right)\right]-\E\left[\mathbf{1}_{Z_{\alpha}>a_n}\right]\right|\right].$
\end{itemize}

First of all, we compute $\lambda_n$ to assure it converges (as $n \rightarrow +\infty$) to a finite value. Then, we compute $b_{1,n}$, $b_{2,n}$ and $b_{3,n}$. Let us start with a preliminary lemma.

\begin{lem}\label{card}
	Considering the previous notations, with straightforward computations we obtain the following results :
	\begin{itemize}
		\item $|\Lpo| \sim p_n^2/2$ as $n \rightarrow +\infty$
		\item $|B_{ij}| \leqslant 8(\tau + K)p_n \sim 8Kp_n$ as $n \rightarrow +\infty$
		\item $|\Lpe| \leqslant |I_K|$
	\end{itemize}
\end{lem}

\subsubsection{Computation of $\lambda_n$}

According to the Chen-Stein method and using the fact that random variables have the same law when indices are in the same set, we have

\begin{eqnarray*}
	\lambda_n & = & \sum\limits_{\alpha \in \Lpo \cup \Lpe}\fP\left(Z_{\alpha}>a_n\right) = \sum\limits_{\alpha \in \Lpo}\fP\left(Z_{\alpha}> a_n\right)+\sum\limits_{\alpha \in \Lpe}\fP\left(Z_{\alpha}>a_n\right){=} |\Lpo|.\fPo+ |\Lpe|.\fPk\\
\end{eqnarray*}
According to theorem \ref{hyp5}, lemma \ref{probIk} and lemma \ref{card}, we have 
$$\underset{n\rightarrow +\infty}{\lim}|\Lpe|.\fPk =0$$ 
while, according to lemma \ref{probI0} and lemma \ref{card}, $$|\Lpo|.\fPo \underset{+\infty}{\sim} p_n^2 \frac{1}{p_n^2}\frac{1}{\sqrt{8\pi}}e^{y/2}=\frac{1}{\sqrt{8\pi}}e^{y/2}.$$ 
Finally, we obtain : 
\begin{equation}
\liminfini\left(\lambda_n\right)=\frac{1}{\sqrt{8\pi}}e^{y/2}
\end{equation}
This quantity appears in the distribution function of the asymptotic Gumbel random variable. 

\subsubsection{Computation of $b_{1,n}$}

We add some notations :
\begin{itemize}
	\item $B_{\alpha}^0 := B_{\alpha} \cap \Lpo$ and $|\Balo| \leqslant |B_{\alpha}| \leqslant 8(\tau+K)p_n$
	\item $B_{\alpha}^K := B_{\alpha} \cap \Lpe$ and $|\BalK|\leqslant K^2$
	\item $\fPal := \fP\left(Z_{\alpha} > a_n\right)$
\end{itemize}
As used above, $Z_{\alpha_1}$ and $Z_{\alpha_2}$ will have the same law as long as $\alpha_1$ and $\alpha_2$ belong to the same set. Then, we have : 
\begin{eqnarray*}
	b_{1,n} & = & \sum\limits_{\alpha \in \Lpo \cup \Lpe}\sum\limits_{\beta \in B_{\alpha}}\fPal\fPb\\
	& = & \sum\limits_{\alpha \in \Lpo}\sum\limits_{\beta \in \Balo}\fPal\fPb+\sum\limits_{\alpha \in \Lpo}\sum\limits_{\beta \in \BalK}\fPal\fPb+\sum\limits_{\alpha \in \Lpe}\sum\limits_{\beta \in \Balo}\fPal\fPb+\sum\limits_{\alpha \in \Lpe}\sum\limits_{\beta \in \BalK}\fPal\fPb\\
	& = & \sum\limits_{\alpha \in \Lpo}\sum\limits_{\beta \in \Balo} \left(\fPo\right)^2+\sum\limits_{\alpha \in \Lpo}\sum\limits_{\beta \in \BalK}\fPo\fPk+\sum\limits_{\alpha \in \Lpe}\sum\limits_{\beta \in \Balo}\fPk\fPo+\sum\limits_{\alpha \in \Lpe}\sum\limits_{\beta \in \BalK}\left(\fPk\right)^2\\
	& = & |\Lpo|.|\Balo|.\left(\fPo\right)^2+|\Lpo|.|\BalK|\fPo\fPk+|\Lpe|.|\Balo|\fPk\fPo+|\Lpe|.|\BalK|.\left(\fPk\right)^2
\end{eqnarray*}
At this point, we need to check that $\liminfini\left(b_{1,n}\right)=0$, so we focus particulary on :
\begin{enumerate}
	\item $|\Lpo|.|\Balo|.\left(\fPo\right)^2$ :
	\begin{eqnarray}
	|\Lpo|.|\Balo|.\left(\fPo\right)^2 & \sim & \frac{1}{2}p_n^2.|\Balo|.\left(\fPo\right)^2 \leqslant 4(\tau+K)p_n^3.O\left(\frac{1}{p_n^4}\right)=O\left(p_n^{\nu - 1}\right)
	\end{eqnarray}
	From assumptions on $\nu$ we have $\liminfini\left[|\Lpo|.|\Balo|.\left(\fPo\right)^2 \right]=0$
	\item $|\Lpo|.|\BalK|\fPo\fPk$ :
	\begin{eqnarray*}
		|\Lpo|.|\BalK|\fPo\fPk & \sim & \frac{1}{2}p_n^2 |\BalK|\fPo\fPk \leqslant \frac{1}{2}K^2p_n^2 \fPo\fPk \leqslant O\left(p_n^{2+2\nu}\right)O\left(p_n^{-2}\right)\fPk =  O\left(p_n^{2\nu}\fPk\right)
	\end{eqnarray*}
	According to lemma \ref{probIk}, we will have $\liminfini \left[p_n^{2\nu}\fPk \right]=0$ iff $2\nu < c_{\gamma}$ which is true from hypothesis $5$ in theorem \ref{main_result}. Then, we obtain :
	$$\liminfini\left[|\Lpo|.|\BalK|\fPo\fPk\right]=0$$
	\item $|\Lpe|.|\Balo|\fPk\fPo$ :
	We use the same principle of computation than previsouly :
	\begin{eqnarray}
	|\Lpe|.|\Balo|\fPk\fPo & \leqslant & p_nK|\Balo|\fPk\fPo \leqslant 8p_nK(\tau + K)p_n\fPk\fPo = O\left(p_n^{2\nu}\fPk\right)
	\end{eqnarray}
	So, from previous assumptions on $\nu$, we have :
	$$\liminfini\left[|\Lpe|.|\Balo|\fPk\fPo\right]=0\,.$$
	\item $|\Lpe|.|\BalK|.\left(\fPk\right)^2$ :
	We have :
	\begin{eqnarray}
	|\Lpe|.|\BalK|.\left(\fPk\right)^2 & \leqslant & pK^3\left(\fPk\right)^2 = O\left(p_n^{1+3\nu}\right)\left(\fPk\right)^2
	\end{eqnarray}
	According to lemma \ref{probIk}, if $1+3\nu < 2c_{\gamma}$, we have 
	$$\liminfini\left[|\Lpe|.|\BalK|.\left(\fPk\right)^2\right]=0\,.$$
\end{enumerate}
To conclude, we finally obtain :
$$\liminfini\left[b_{1,n}\right]=0\,.$$
\begin{rem}\label{rem}
	The main constraint here is $p_nK^3\left(\fPk\right)^2\to 0$ which is true from condition $p_nK\fPk\to 0$ of remark \ref{rem1}.
\end{rem}

\subsubsection{Computation of $b_{2,n}$}

The computation of $b_{2,n}$ is the most technical part. As we did for the computation of $b_{1,n}$, we will divide this computation into four parts (according on which set we are). We remind the definition of $b_{2,n}$ :
$$b_{2,n}=\sum\limits_{\alpha\in \Lpo \cup \Lpe}\sum\limits_{\beta\in B_{\alpha}}\fP\left(Z_{\alpha}>a_n, Z_{\beta}>a_n\right)\,.$$
Here we introduce some new notations :
\begin{itemize}
	\item $\fP_{\alpha\beta} := \fP\left(Z_{\alpha}>a_n,Z_{\beta}>a_n\right)$
	\item $\fP_{0i} := \fP_{\alpha\beta}\mathbf{1}_{\alpha \in \Lpo}\mathbf{1}_{\beta \in \Omega_i}$ where $\Omega_i$ will be a subset of indices and $i$ an integer.
	\item $\fP_{Ki} := \fP_{\alpha\beta}\mathbf{1}_{\alpha \in \Lpe}\mathbf{1}_{\beta \in \Omega_i}$ where $\Omega_i$ will be a subset of indices and $i$ an integer.
\end{itemize}
To show that $\liminfini[b_{2,n}]=0$, we will divide it into four sums, each one being the sum of the same probability on a given set of indices. Then, we have : 

\begin{equation}
b_{2,n}=
\underbrace{\sum\limits_{\alpha\in \Lpo}\sum\limits_{\beta\in \Balo}\fP_{\alpha\beta}}_{:=Q_1}
+ \underbrace{\sum\limits_{\alpha\in \Lpo}\sum\limits_{\beta\in \BalK}\fP_{\alpha\beta}}_{:=Q_2}
+ \underbrace{\sum\limits_{\alpha\in \Lpe}\sum\limits_{\beta\in \Balo}\fP_{\alpha\beta}}_{:=Q_3}
+ \underbrace{\sum\limits_{\alpha\in \Lpe}\sum\limits_{\beta\in \BalK}\fP_{\alpha\beta}}_{:=Q_4}
\end{equation}

\underline{\textbf{Computation of $Q_1$ :}}\label{calculQ1}
\newline

First, we define some additional subsets of indices. In particular, we have : 
\begin{enumerate}
	\item $\Omega_1 := \{(u,v)\in\Lpo : i-u < \tau \text{  and  } j-v < \tau \}$ and $|\Omega_1| \leqslant \tau^2$
	\item $\Omega_2 := \{(u,v)\in\Lpo : i-u < \tau \text{  and  } \tau < j-v < \tau +K \}$ and $|\Omega_2| \leqslant \tau K$
	\item $\Omega_3 := \{(u,v)\in\Lpo :  \tau < i-u < \tau + K \text{  and  } j-v < \tau \}$ and $|\Omega_3| \leqslant \tau K$
	\item $\Omega_4 := \{(u,v)\in\Lpo : i-u < \tau \text{  and  } \tau + K \leqslant j-v \}$ and $|\Omega_4| \leqslant \tau\left(p_n-\tau-K\right) \leqslant \tau p_n$
	\item $\Omega_5 := \{(u,v)\in\Lpo : \tau + K \leqslant i-u \text{  and  } j-v < \tau \}$ and $|\Omega_5| \leqslant \tau\left(p_n-\tau-K\right) \leqslant \tau p_n$
	\item $\Omega_6 := \{(u,v)\in\Lpo : \tau < i-u < \tau + K \text{  and  } \tau < j-v < \tau +K \}$ and $|\Omega_6| \leqslant K^2$
	\item $\Omega_7 := \{(u,v)\in\Lpo : \tau < i-u < \tau+ K <\text{  and  } \tau + K \leqslant j-v \}$\\ and $|\Omega_7| \leqslant K \left(p_n-\tau-K\right) \leqslant Kp_n$
	\item $\Omega_8 := \{(u,v)\in\Lpo : \tau + K \leqslant i-u \text{  and  } j-v < \tau \}$ and $|\Omega_8| \leqslant K \left(p_n-\tau-K\right) \leqslant Kp_n$
\end{enumerate}
We have  : 
\begin{equation}
Q_1 \leqslant 4\sum\limits_{i=1}^{8}\sum\limits_{\alpha\in\Lpo}\sum\limits_{\beta\in\Omega_i}\fP_{\alpha\beta}
\end{equation}
Then, using the fact that on each given subset the random variables have the same law : 

\begin{eqnarray}
Q_1 & \leqslant &  
\left|\Lpo\right|.\left|\Omega_1\right|\fP_{01}+
\left|\Lpo\right|.\left|\Omega_2\right|\fP_{02}+
\left|\Lpo\right|.\left|\Omega_3\right|\fP_{03}+
\left|\Lpo\right|.\left|\Omega_4\right|\fP_{04}\\
& & +
\left|\Lpo\right|.\left|\Omega_5\right|\fP_{05}+
\left|\Lpo\right|.\left|\Omega_6\right|\fP_{06}+
\left|\Lpo\right|.\left|\Omega_7\right|\fP_{07}+
\left|\Lpo\right|.\left|\Omega_8\right|\fP_{08}
\end{eqnarray}

So, we just have to show that each part will have a null limit when $n$ is going to infinity.

\begin{lem}\label{cas11}
	Using the previous notations, we have, as $n \rightarrow +\infty$ :
	\begin{eqnarray}
	\left|\Lpo\right|.\left|\Omega_1\right|\fP_{01}\rightarrow 0
	\end{eqnarray} 
\end{lem}

\begin{proof}
	We have :
	\begin{eqnarray*}
		\left|\Lpo\right|.\left|\Omega_1\right|\fP_{01} & \leqslant &\left|\Lpo\right| \tau^2\fP_{01}\sim \frac{1}{2}p_n^2\tau^2\fP_{01}=o\left(p_n^{2+2t}\fP_{01}\right) \text{  for any  } t>0\\
	\end{eqnarray*} 
	where we use \cref{card} for the equivalent. We can write : 
	\begin{eqnarray}
	\fP_{01} & = & \fP\left(\left|\sum\limits_{k=1}^{n}u_k^1u_k^2\right|>a_n,\left|\sum\limits_{k=1}^{n}u_k^3u_k^4\right|>a_n\right)\\
	\end{eqnarray}
	where $\left(u_k^1,u_k^2,u_k^3,u_k^4\right)_{1\leqslant k \leqslant n} \overset{i.i.d}{\sim} \mathcal{N}_{4}\left(0,\Sigma_4\right)$ and 
	\[\Sigma_4=\left(\begin{array}{cccc}
	1 & 0 & r_1 & 0 \\ 
	0 & 1 & 0 & r_2 \\ 
	r_1 & 0 & 1 & 0 \\ 
	0 & r_2 & 0 & 1
	\end{array}\right)\,,\]
	where coefficients $r_1,r_2$ are from the correlation matrix $(r_{kj})$.
	From Lemma $6.11$ of \cite{CaiJiang}, focusing on equation $(131)$, we know that
	\begin{eqnarray}
	\fP_{01} & \leqslant & O\left( p_n^{-2b^2+\varepsilon_1} \right) + O\left( p_n^{-2-2c^2+\varepsilon_2} \right) \text{ as } n \rightarrow +\infty
	\end{eqnarray}
	for any $\varepsilon_1,\varepsilon_2 >0$ and where $a = \frac{1 + (1-\delta)^2}{2}$, $b=\frac{a}{(1-\delta)^2}$ and $c=\frac{1-a}{3}$ for $\delta \in ]0,1[$.
	By construction $b^2-1>0$, hence for a well-chosen $t$ such that $t < b^2-1$, there exists $\varepsilon_1(\delta)>0$ such that we have :
	\begin{eqnarray}\label{int1}
	\varepsilon_1 < 2b^2-2-2t\,.
	\end{eqnarray} 
	Analogously, we can find $\varepsilon_2(\delta)$ such that:
	\begin{eqnarray}
	\varepsilon_2 < 2\left(c^2-t\right)\,.
	\end{eqnarray}
	Since $\tau =o(p^t)$ for any $t>0$, we have	\begin{eqnarray}
	\left|\Lpo\right|.\left|\Omega_1\right|\fP_{01} \rightarrow 0 \text{ as } n \rightarrow +\infty\,.
	\end{eqnarray}
\end{proof}

\begin{lem}\label{cas12}
	Using previous notations, we have :
	\begin{eqnarray}
	\left|\Lpo\right|.\left|\Omega_2\right|\fP_{02}\rightarrow 0
	\end{eqnarray} 
\end{lem}

\begin{proof}
	We have :
	\begin{eqnarray*}
		\left|\Lpo\right|.\left|\Omega_2\right|\fP_{02} & \leqslant & \tau K\left|\Lpo\right|\fP_{02}\sim \frac{1}{2}p_n^2\tau K\fP_{02}=O\left(\tau p_n^{2+\nu}\fP_{02}\right)\\
	\end{eqnarray*} 
	where we use the lemma \ref{card} for the equivalence above. In this proof, we almost have the same case than in the proof of lemma \ref{cas11}. In fact, the only difference is the matrix $\Sigma_4$ which is now \[\Sigma_4=\left(\begin{array}{cccc}
	1 & 0 & r & 0 \\ 
	0 & 1 & 0 & \eps_n \\ 
	r & 0 & 1 & 0 \\ 
	0 & \eps_n & 0 & 1
	\end{array}\right)\,,\]  
	where $r$ is a coefficient from the matrix $(r_{kj})$.
	So, by the same method we have 
	\begin{eqnarray}
	p_n^{2+\nu}\fP_{02} \rightarrow 0 
	\end{eqnarray}
	iff $\varepsilon_1 < 2b^2-2-\nu$ and $\varepsilon_2 < 2c^2-\nu$ where we still have $b=\frac{1+\left(1-\delta\right)^2}{2\left(1-\delta\right)^2}$ and $c=\frac{1-\left(1-\delta\right)^2}{6}$. Moreover we can show that $b^2-1>c^2$. Then, if $\nu < 2c^2$ (fullfilled by assumptions in theorem \ref{main_result}), and from $\tau=o\left(p_n^t\right)$ for any $t>0$, we have :
	\begin{eqnarray}
	\left|\Lpo\right|.\left|\Omega_2\right|\fP_{02}\rightarrow 0\,.
	\end{eqnarray} 
\end{proof}

\begin{lem}\label{cas13}
	Using notations previously introduced, we have :
	\begin{eqnarray}
	\left|\Lpo\right|.\left|\Omega_3\right|\fP_{03}\rightarrow 0
	\end{eqnarray} 
\end{lem}

\begin{proof}
	This proof is exactly the same than for lemma \ref{cas12} except that the matrix becomes 
	\[\Sigma_4=\left(\begin{array}{cccc}
	1 & 0 & \eps_n & 0 \\ 
	0 & 1 & 0 & r \\ 
	\eps_n & 0 & 1 & 0 \\ 
	0 & r & 0 & 1
	\end{array}\right)\,.\] In particular, we obtain the same condition on $\nu$.
\end{proof}

\begin{lem}\label{cas14}
	Using notations previously introduced, we have :
	\begin{eqnarray}
	\left|\Lpo\right|.\left|\Omega_4\right|\fP_{04}\rightarrow 0
	\end{eqnarray} 
\end{lem}

\begin{proof}
	We have :
	\begin{eqnarray}
	\left|\Lpo\right|.\left|\Omega_4\right|\fP_{04} \leqslant\tau p_n \left|\Lpo\right|\fP_{04}\sim \tau p_n^3\fP_{04}
	\end{eqnarray}
	Now, the correlation matrix is $\Sigma_4=\left(\begin{array}{cccc}
	1 & 0 & r & 0 \\ 
	0 & 1 & 0 & 0 \\ 
	r & 0 & 1 & 0 \\ 
	0 & 0 & 0 & 1
	\end{array}\right)$.\\
	Thanks to the lemma 6.9 in \cite{CaiJiang}, proved in the supplementary paper, we obtain $\fP_{04}=O\left(p_n^{-4+\eps}\right)$ for any $\eps >0$. Then, we have $p_n^3 \tau \fP_{04}=O\left(\frac{\tau}{p_n^{1-\eps}}\right)$ which tends to $0$ as $n\to\infty$ since $\tau = o\left(p_n^t\right)$ for any $t>0$.
\end{proof}

\begin{lem}\label{cas15}
	Using notations previously introduced, we have :
	\begin{eqnarray}
	\left|\Lpo\right|.\left|\Omega_5\right|\fP_{05}\rightarrow 0\,.
	\end{eqnarray} 
\end{lem}

\begin{proof}
	This proof is exactly the same than for lemma \ref{cas14} considering the correlation matrix 
	\[\Sigma_4=\left(\begin{array}{cccc}
	1 & 0 & 0 & 0 \\ 
	0 & 1 & 0 & r \\ 
	0 & 0 & 1 & 0 \\ 
	0 & r & 0 & 1
	\end{array}\right)\]
\end{proof}

\begin{lem}\label{cas16}
	Using notations previously introduced, we have :
	\begin{eqnarray}
	\left|\Lpo\right|.\left|\Omega_6\right|\fP_{06}\rightarrow 0
	\end{eqnarray} 
\end{lem}

\begin{proof} 
	This proof is exactly the same than for lemma \ref{cas12} except that the matrix become 
	\[\Sigma_4=\left(\begin{array}{cccc}
	1 & 0 & \eps_n & 0 \\ 
	0 & 1 & 0 & \eps_n \\ 
	\eps_n & 0 & 1 & 0 \\ 
	0 & \eps_n & 0 & 1
	\end{array}\right)\] In particular, we have : 
	\begin{eqnarray}
	\left|\Lpo\right|.\left|\Omega_6\right|\fP_{06}\leqslant K^2\left|\Lpo\right|\fP_{06} \sim O\left(p_n^{2+2\nu}\fP_{06}\right)
	\end{eqnarray}
	with $\Sigma_4$ as correlation matrix for the $4$-uplet in $\fP_{06}$. As for lemma \ref{cas12}, we have the following conditions
	\begin{eqnarray}
	\eps_1 < 2b^2-2-2\nu \text{    and    } \eps_2< 2\left(c^2 - \nu\right)
	\end{eqnarray}
	which is summarized in $\nu < c^2$, and which is true considering theorem \ref{main_result}. Then, we obtain the desired result :
	\begin{eqnarray}
	\left|\Lpo\right|.\left|\Omega_6\right|\fP_{06}\rightarrow 0
	\end{eqnarray} 
\end{proof}

\begin{lem}\label{cas17}
	Using notations previously introduced, we have :
	\begin{eqnarray}
	\left|\Lpo\right|.\left|\Omega_7\right|\fP_{07}\rightarrow 0
	\end{eqnarray} 
\end{lem}

\begin{proof}
	This proof is exactly the same than for lemma \ref{cas14} considering the correlation matrix 
	\[\Sigma_4=\left(\begin{array}{cccc}
	1 & 0 & \eps_n & 0 \\ 
	0 & 1 & 0 & 0 \\ 
	\eps_n & 0 & 1 & 0 \\ 
	0 & 0 & 0 & 1
	\end{array}\right)\,.\]
\end{proof}

\begin{lem}\label{cas18}
	Using notations previously introduced, we have :
	\begin{eqnarray}
	\left|\Lpo\right|.\left|\Omega_8\right|\fP_{08}\rightarrow 0
	\end{eqnarray} 
\end{lem}

\begin{proof}
	This proof is exactly the same than for lemma \ref{cas14} considering the correlation matrix 
	\[\Sigma_4=\left(\begin{array}{cccc}
	1 & 0 & 0 & 0 \\ 
	0 & 1 & 0 & \eps_n \\ 
	0 & 0 & 1 & 0 \\ 
	0 & \eps_n & 0 & 1
	\end{array}\right)\,.\]
\end{proof}

\begin{rem}\label{rem3}
	The main constraint here for $Q_1$ is $\nu<c^2$.
\end{rem}

\vspace{0.5cm}
\underline{\textbf{Computation of $Q_2$ :}}
\vspace{0.5cm}
\\
For this case, we will divide the computation into two parts. Indeed, we will consider two cases : when $\alpha$ is close to the set $\Lpe$ and when it is not. For that purpose, we introduce the following sets : 

$$I_{0,I} = \{(i,j)\in[\![1,p]\!], i < j \text{ and } \tau + K < j - i < \tau + 4K \} \text{   and   } \LpoI = \Lpo \cap I_{0,I}$$ 

and 

$$I_{0,II}=I_0\backslash I_{0,I} \text{   and   } \LpoII = \Lpo \cap I_{0,II}\,.$$ 

We can write :

$$ Q_2 :=\sum\limits_{\alpha\in \Lpo}\sum\limits_{\beta\in \BalK}\fP_{\alpha\beta}=\sum\limits_{\alpha \in \LpoI}\sum\limits_{\beta\in \BalK}\fP_{\alpha\beta}+\sum\limits_{\alpha \in \LpoII}\sum\limits_{\beta\in \BalK}\fP_{\alpha\beta}\,.$$

Now, we look at the sum on $\LpoI$. We notice that on this set, the probability $\fP_{\alpha\beta}$ is issued from a Gaussian vector with correlation matrix 
\[\Sigma_4 =\left(\begin{array}{cccc}
1 & 0 & r_1 & r_2 \\ 
0 & 1 & r_3 & r_4 \\ 
r_1 & r_3 & 1 & \eps_n \\ 
r_2 & r_4 & \eps_n & 1
\end{array}\right)\] where $|r_i| \leqslant 1-\delta $ for all $i\in \{1,2,3,4\}$. Moreover, coefficients $(r_i)_{i}$ may be replaced here by $\eps_n$ according to the position of the indice in both sets $\LpoI$ and $\BalK$. But we know that $\liminfini(\eps_n)=0$ then, for $n$ large enough, we still have $|r_i| \leqslant 1- \delta$. Using Cauchy-Schwarz inequality, we can bound : 

$$\fP_{\alpha\beta}=\E\left[\mathbf{1}_{Z_{\alpha}>a_n}\mathbf{1}_{Z_{\beta}>a_n}\right] \leqslant \sqrt{\E\left[\mathbf{1}^2_{Z_{\alpha}>a_n}\right]\E\left[\mathbf{1}^2_{Z_{\beta}>a_n}\right]} \leqslant \sqrt{\E\left[\mathbf{1}_{Z_{\alpha}>a_n}\right]\E\left[\mathbf{1}_{Z_{\beta}>a_n}\right]}=\sqrt{\fP_{\alpha}\fP_{\beta}}$$

Now, we use the fact that $\alpha\in\LpoI \subset \Lpo$ and $\beta\in \BalK \subset I_K$, then : 

$$\fP_{\alpha\beta} \leqslant \sqrt{\fPo\fPk}$$

which is true for any $|r_i| \leqslant 1$ then, $\underset{|r_i|\leqslant 1, i=1,\dots, 4}{\sup}\fP_{\alpha\beta} \leqslant \sqrt{\fPo\fPk}$. At this point, using $|\LpoI| \leqslant 3Kp$, $|\BalK| \leqslant K^2$, and $\fPo = O\left(p^{-2}\right)$, we get:

$$\sum\limits_{\alpha \in \LpoI}\sum\limits_{\beta\in \BalK}\fP_{\alpha\beta} \leqslant 3K^3p\sqrt{\fPo\fPk} \leqslant 3K^3\fPk^{1/2}O(1)=O\left(p^{3\nu}\fP^{1/2}\right) \text{ as } n \rightarrow +\infty.$$

Now, using lemma \ref{probIk}:

\begin{equation}\label{eq62}
p^{3\nu}\fP^{1/2} \underset{n\rightarrow +\infty}{\longrightarrow}0 \Leftrightarrow 3\nu < \frac{1}{2}\left(\frac{1}{2}\gamma^2-2|\gamma|+2\right) \Leftrightarrow \nu < \frac{1}{6}c_{\gamma}\end{equation}

which is true according to assumptions of theorem \ref{main_result}.\\

Now, let us focuse on the computation of $\LpoII$. For that purpose, we introduce four subsets: 

\begin{itemize}
	\item $\Omega_1^2 := \{(u,v) \in \BalK : u - i < \tau \text{ and } j - v > \tau + K  \}$ and $|\Omega_1^2| \leqslant K\tau$
	\item $\Omega_2^2 := \{(u,v) \in \BalK : \tau + K < u - i \text{ and } j - v < \tau   \}$ and $|\Omega_2^2| \leqslant K\tau$
	\item $\Omega_3^2 := \{(u,v) \in \BalK : \tau \leqslant u - i \leqslant \tau + K \text{ and } j - v > \tau + K  \}$ and $|\Omega_2^2| \leqslant K^2$
	\item $\Omega_4^2 := \{(u,v) \in \BalK : \tau + K <  u - i  \text{ and } \tau \leqslant j - v \leqslant \tau + K  \}$ and $|\Omega_4^2| \leqslant K^2$
\end{itemize}

We have : 

$$\sum\limits_{\alpha \in \LpoII}\sum\limits_{\beta\in \BalK}\fP_{\alpha\beta} \leqslant 4\sum\limits_{i=1}^{4}\sum\limits_{\alpha \in \LpoII}\sum\limits_{\beta\in \Omega_i^2}\fP_{\alpha\beta}$$.

In order to consider all these subset, we have the four next lemmas : 

\begin{lem}\label{lemQ2}
	Considering the same notations as previously : $\left(u_k^1,u_k^2,u_k^3,u_k^4\right)_{1\leqslant k \leqslant n} \overset{i.i.d}{\sim} \mathcal{N}_{4}\left(0,\Sigma_4\right)$. If the probability $\fP_{\alpha\beta}$ is issued from a Gaussian vector with covariance matrix $\Sigma_4 =\left(\begin{array}{cccc}
	1 & 0 & r_1 & x \\ 
	0 & 1 & 0 & 0 \\ 
	r_1 & 0 & 1 & \eps_n \\ 
	x & 0 & \eps_n & 1
	\end{array}\right)$ where $x \in \{\eps_n,0\}$, then : 
	\begin{equation}
	\sum\limits_{\alpha\in\LpoII}\sum\limits_{\beta\in\Omega_{1}^2}\fP_{\alpha\beta} \leqslant O\left(p^{t+\nu+\epsilon}\fPk^{1/2}\right)
	\end{equation}
	for any $t>0$ and any $\epsilon >0$.
\end{lem}

\begin{proof}
	In order to prove this result, we observe that in this case, for all $k \geqslant 1$, $u_k^2$ is independent of $\{u_k^1,u_k^3,u_k^4\}$. It means that conditionally on $u_k^1$, we have independence between $Z_{12}$ and $Z_{34}$. By consequence, using Cauchy-Schwarz, we obtain : 
	\begin{eqnarray}
	\fP_{\alpha\beta} & = & \E\left[\E\left[\mathbf{1}_{Z_{12}>a_n}\mathbf{1}_{Z_{34}>a_n}|u_k^1, k=1, \dots, n\right]\right]\\
	& = & \E\left[\E\left[\mathbf{1}_{Z_{12}>a_n}|u_k^1, k=1, \dots, n\right]\E\left[\mathbf{1}_{Z_{34}>a_n}|u_k^1, k=1, \dots, n\right]\right]\\
	& \leqslant & \sqrt{\E\left[\E\left[\mathbf{1}_{Z_{12}>a_n}|u_k^1, k=1, \dots, n\right]^2\right]\E\left[\E\left[\mathbf{1}_{Z_{34}>a_n}|u_k^1, k=1, \dots, n\right]^2\right]}
	\end{eqnarray}
	Now, because $u_k^1$ is independent of $u_k^2$, we can use lemma 6.7 from \cite{CaiJiang} and we have :
	
	$$\E\left[\E\left[\mathbf{1}_{Z_{12}>a_n}|u_k^1, k=1, \dots, n\right]^2\right]=O\left(p^{-4+\epsilon}\right)$$ 
	
	for any $\epsilon>0$. And on the other side, we have :
	
	$$\E\left[\E\left[\mathbf{1}_{Z_{34}>a_n}|u_k^1, k=1, \dots, n\right]^2\right]\leqslant \fPk$$
	
	Finally, using $|\LpoII|\leqslant p^2$ and $|\Omega_1^2| \leqslant K\tau$, and writing $K=O(p^{\nu})$, we have the desired result :
	\begin{equation}
	\sum\limits_{\alpha\in\LpoII}\sum\limits_{\beta\in\Omega_{1}^2}\fP_{\alpha\beta} \leqslant O\left(p^{t+\nu+\epsilon}\fPk^{1/2}\right)
	\end{equation}
\end{proof}

Now, from lemma \ref{lemQ2} we have the condition : 
$$t + \nu + \epsilon < \frac{1}{2}\left(\frac{1}{2}\gamma^2-2|\gamma|+2\right)\,,$$
which can be fullfilled from condition  in \cref{eq62} and for well-chosen $t>0$ and $\epsilon >0$. Finally we obtain, with our condition on $\nu$ that :

$$\liminfini\left[\sum\limits_{\alpha\in\LpoII}\sum\limits_{\beta\in\Omega_{1}^2}\fP_{\alpha\beta}\right]=0\,.$$

For the other subsets $\Omega_i^2$, $i=2,3,4$, we will use the same method. Indeed we notice that respectively for $\Omega_2^2$, $\Omega_3^2$ and $\Omega_4^2$, the covariance matrices involved are respectively : 

\begin{equation}
\Sigma_4^2 =\left(\begin{array}{cccc}
1 & 0 & 0 & 0 \\ 
0 & 1 & x & r \\ 
0 & x & 1 & \eps_n \\ 
0 & r & \eps_n & 1
\end{array}\right)
\text{   ,  } 
\Sigma_4^3 =\left(\begin{array}{cccc}
1 & 0 & \eps_n & x \\ 
0 & 1 & 0 & 0 \\ 
\eps_n & 0 & 1 & \eps_n \\ 
x & 0 & \eps_n & 1
\end{array}\right)
\text{   ,  }
\Sigma_4^4 =\left(\begin{array}{cccc}
1 & 0 & 0 & 0 \\ 
0 & 1 & x & \eps_n \\ 
0 & x & 1 & \eps_n \\ 
0 & \eps_n & \eps_n & 1
\end{array}\right)
\end{equation}

For each case, we use the fact that we always have $u_k^2$ (or $u_k^1$) independent of the other three random variables. Also, in order to use the \cref{lemQ2}, we notice that by construction $Z_{12}=Z_{21}$. Then, for cases $\Omega_3^2$ and $\Omega_4^2$, we just have to consider the Gaussian vector $(u_k^2,u_k^1,u_k^3,u_k^4)$ instead of $(u_k^1,u_k^2,u_k^3,u_k^4)$. In that way, all matrices have the same form than in lemma \ref{lemQ2} and similar upper-bound for the probability. Now, we use the upper-bound of subsets $\Omega_i^2$. More precisely, for $\Omega_3^2$ and $\Omega_4^2$:
$$\sum\limits_{\alpha\in\LpoII}\sum\limits_{\beta\in\Omega_{3}^2}\fP_{\alpha\beta} \leqslant |\LpoII|.|\Omega_3^2|\fPk^{1/2}O\left(p^{-2+\epsilon}\right)= O\left(p^{2\nu +\epsilon}\fPk^{1/2}\right) \text{ for any } \epsilon >0$$
It means that we need to have, according to lemma \ref{probIk} : 
\begin{equation}\label{eq70}
2\nu  < \frac{1}{2}\left(\frac{1}{2}\gamma^2-2|\gamma|+2\right) \Leftrightarrow \nu  < \frac{1}{4}\left(\frac{1}{2}\gamma^2-2|\gamma|+2\right)\,.
\end{equation}
With this condition
$$\liminfini\left[\sum\limits_{\alpha\in\LpoII}\sum\limits_{\beta\in\Omega_{3}^2}\fP_{\alpha\beta}\right]=0$$
and 
$$\liminfini\left[\sum\limits_{\alpha\in\LpoII}\sum\limits_{\beta\in\Omega_{4}^2}\fP_{\alpha\beta}\right]=0$$
Finally, the case $\Omega_2^2$ leads to the exactly same result than for $\Omega_1^2$ because of the upper-bound on $|\Omega_2^2|$ which is the same than for $|\Omega_1^2|$. Then, we have : 
$$\liminfini\left[\sum\limits_{\alpha \in \LpoII}\sum\limits_{\beta\in \BalK}\fP_{\alpha\beta}\right]=0$$ 
and then 
$$\liminfini\left[Q_2\right]=0$$
\begin{rem}\label{rem4}
	The main constaint here is $\nu<\frac{1}{6}c_{\gamma}$.
\end{rem}
\vspace{0.5cm}
\underline{ \textbf{Computation of $Q_3$ :}}
\vspace{0.5cm}

We focuse here on $Q_3 = \sum\limits_{\alpha\in\Lpe}\sum\limits_{\beta\in\Balo}\fP_{\alpha\beta}$. Once more, we will consider different subsets for $\beta$ according to its place into $\Balo$. More precisely, let us define : 

\begin{itemize}
	\item $\Omega_1^3=\{(u,v)\in\Lpo : i-u < \tau \text{ and } v-j < \tau \}$ and $|\Omega_1^3|\leqslant \tau^2$.
	\item $\Omega_2^3=\{(u,v)\in\Lpo : i-u < \tau \text{ and } \tau \leqslant v-j \leqslant  \tau+K \}$ and $|\Omega_2^3|\leqslant \tau K$.	
	\item $\Omega_3^3=\{(u,v)\in\Lpo : i-u < \tau \text{ and } \tau + K < v - j \}$ and $|\Omega_3^3|\leqslant \tau p$.
	\item $\Omega_4^3=\{(u,v)\in\Lpo : \tau \leqslant i-u \leqslant \tau+K \text{ and } v-j < \tau \}$ and $|\Omega_4^3|\leqslant \tau K$.
	\item $\Omega_5^3=\{(u,v)\in\Lpo : \tau + K < i-u \text{ and } v-j < \tau \}$ and $|\Omega_5^3|\leqslant \tau p$.
	\item $\Omega_6^3=\{(u,v)\in\Lpo : \tau \leqslant i-u \leqslant \tau+K  \text{ and } \tau \leqslant v-j \leqslant  \tau+K \}$ and $|\Omega_6^3|\leqslant K^2$.
	\item $\Omega_7^3=\{(u,v)\in\Lpo : \tau \leqslant i-u \leqslant \tau+K  \text{ and } \tau + K < v - j \}$ and $|\Omega_7^3|\leqslant Kp$.
	\item $\Omega_8^3=\{(u,v)\in\Lpo : \tau + K < i-u  \text{ and } \tau \leqslant v-j \leqslant  \tau+K \}$ and $|\Omega_8^3|\leqslant Kp$.
\end{itemize}

Then, we have : 
$$Q_3 \leqslant \sum\limits_{i=1}^{8}\sum\limits_{\alpha\in\Lpe}\sum\limits_{\beta\in\Omega_i^3}\fP_{\alpha\beta}$$

For $Q_3$, we  use the computation of $Q_2$. Indeed, covariance matrices which are involved in the computation fo $Q_3$ are similar than for $Q_2$. The similarity comes from the fact that we exchange the role between $\alpha$ and $\beta$. More precisely, for $Q_2$ we had $\alpha\in\Lpo$ and $\beta\in\BalK$ and now we have $\alpha\in\Lpe$ and $\beta\in\Balo$. So, covariance matrices here will have the same structure than in $Q_2$ exchanging columns $\{1,2\}$ and columns $\{3,4\}$.\\

We notice that : 
\begin{equation}\label{majQ3}
|\Omega_1^3|,|\Omega_2^3|,|\Omega_4^3|\leqslant |\Omega_6^3|\leqslant K^2
\end{equation} Using the fact that $\Balo \subset \Lpo$, we have :
$$\sum\limits_{\alpha\in\Lpe}\sum\limits_{\beta\in\Omega_6^3}\fP_{\alpha\beta}\leqslant \sum\limits_{\alpha\in\Lpe}\sum\limits_{\beta\in\Omega_6^3}\fPo \leqslant pK.K^2\fPo =O\left(p^{3\nu-1}\right) \text{ as } n \rightarrow +\infty$$

However, from condition in \cref{eq62}, we have $\nu < \frac{1}{3}\left(\frac{1}{4}\gamma^2-|\gamma|+1\right)$. But, $\gamma\in ]-2+\sqrt{2},2-\sqrt{2}[$. Then, $\frac{1}{3}\left(\frac{1}{4}\gamma^2-|\gamma|+1\right) \in ]\frac{1}{6},\frac{1}{3}[$. It leads, in particular, to $\nu < \frac{1}{3}$ and then : 
$$\liminfini\left[\sum\limits_{\alpha\in\Lpe}\sum\limits_{\beta\in\Omega_6^3}\fP_{\alpha\beta}\right]=0$$

And so, from equation \cref{majQ3} : 
\begin{equation}
\liminfini\left[\sum\limits_{\alpha\in\Lpe}\sum\limits_{\beta\in\Omega_1^3}\fP_{\alpha\beta}\right]=\liminfini\left[\sum\limits_{\alpha\in\Lpe}\sum\limits_{\beta\in\Omega_2^3}\fP_{\alpha\beta}\right]=\liminfini\left[\sum\limits_{\alpha\in\Lpe}\sum\limits_{\beta\in\Omega_4^3}\fP_{\alpha\beta}\right]=0
\end{equation}

Now, we look at the case when $\beta$ belongs to $\Omega_3^3,\Omega_5^3,\Omega_7^3$ and $\Omega_8^3$. We start by describing each covariance matrix involved. We note $x \in \{\eps_n,0\}$ and $r$ a correlation coefficient such that $|r|\leqslant 1-\delta$. We have : 
\begin{itemize}
	\item[$\bullet$] for $\alpha\in\Lpe, \beta\in\Omega_3^3$, $\Sigma_4^3 =\left(\begin{array}{cccc}
	1 & \eps_n & r & 0 \\ 
	\eps_n & 1 & x & 0 \\ 
	r & x & 1 & 0 \\ 
	0 & 0 & 0 & 1
	\end{array}\right)$
	\item[$\bullet$] for $\alpha\in\Lpe, \beta\in\Omega_5^3$, $\Sigma_4^3 =\left(\begin{array}{cccc}
	1 & \eps_n & 0 & x \\ 
	\eps_n & 1 & 0 & r \\ 
	0 & 0 & 1 & 0 \\ 
	x & r & 0 & 1
	\end{array}\right)$
	\item[$\bullet$] for $\alpha\in\Lpe, \beta\in\Omega_7^3$, $\Sigma_4^3 =\left(\begin{array}{cccc}
	1 & \eps_n & \eps_n & 0 \\ 
	\eps_n & 1 & x & 0 \\ 
	\eps_n & x & 1 & 0 \\ 
	0 & 0 & 0 & 1
	\end{array}\right)$
	\item[$\bullet$] for $\alpha\in\Lpe, \beta\in\Omega_8^3$, $\Sigma_4^3 =\left(\begin{array}{cccc}
	1 & \eps_n & 0 & x \\ 
	\eps_n & 1 & 0 & \eps_n \\ 
	0 & 0 & 1 & 0 \\ 
	x & \eps_n & 0 & 1
	\end{array}\right)$.\\
\end{itemize}
We observe that for each case above one variable $u_k^3$ or $u_k^4$ is independent from the three other ones. Then we can use the same method than for $Q_2$ (conditioning on $u_k^3$ when $u_k^4$ is independent of the other ones or reversely on $u_k^3$).Then, by Cauch-Schwarz, we obtain the same upper-bound for $\fP_{\alpha\beta}$. To show that we obtain the desired convergence, we study here the worst case. That is to say, using the fact that $|\Omega_i^3| \leqslant Kp$ for $i=3,5,7,8$, we can write :

$$\sum\limits_{\alpha\in\Lpe}\sum\limits_{\beta\in\Omega_7^3}\fP_{\alpha\beta}\leqslant |\Lpe|.|\Omega_7^3|\fPk^{1/2}O\left(p^{-2+\epsilon}\right) \leqslant O\left(p^{2\nu+\epsilon}\fPk\right) \text{ as } n \rightarrow +\infty$$

Then, we have exactly the same condition on $\nu$ that for equation \cref{eq70}. It means that 
$$\liminfini\left[\sum\limits_{\alpha\in\Lpe}\sum\limits_{\beta\in\Omega_7^3}\fP_{\alpha\beta}\right]=0,$$ 
and it induces that 
$$\liminfini\left[Q_3\right]=0.$$

\pagebreak
\underline{ \textbf{Computation of $Q_4$ :}}
\medskip\par

This last quantity is simpler because we can write : 

$$Q_4 = \sum\limits_{\alpha\in\Lpe}\sum\limits_{\beta\in\BalK}\fP_{\alpha\beta}\leqslant \sum\limits_{\alpha\in\Lpe}\sum\limits_{\beta\in\BalK}\fPk = \left|\Lpe\right|.\left|\BalK\right|\fPk \leqslant pK^3\fP_k = O\left(p^{1+3\nu}\fPk\right) \text{ as } n \rightarrow +\infty$$

So, to have $\liminfini\left(Q_4\right)=0$ we must have, according to \cref{probIk}: 

\begin{eqnarray*}
	1+3\nu < \frac{1}{2}\gamma^2-2|\gamma|+2 & \Leftrightarrow & \nu <\frac{1}{3}( c_{\gamma}-1)\\
\end{eqnarray*}
Under this assumption on $\nu$ we have
$$\liminfini\left[Q_4\right]=0$$

Finally, gathering the results from $Q_1$ to $Q_4$ and under sufficient assumtions, we have : 

$$\liminfini\left[b_{2,n}\right]=0.$$

\begin{rem}\label{rem5}
	The constraint on $\nu$ is $\nu<\frac{1}{3}(c_{\gamma}-1)$.
\end{rem}

\subsubsection{Computation of $b_{3,n}$}

We have :
\begin{multline}
b_{3,n} =  \sum\limits_{\alpha \in \Lpo \cup \Lpe}\E\left[\left|\E\left[\mathbf{1}_{Z_{\alpha}>a_n}|\sigma\left(Z_{\beta}, \beta \in (\Lpo \cup \Lpe)\backslash B_{\alpha}\right)\right]-\E\left[\mathbf{1}_{Z_{\alpha}>a_n}\right]\right|\right]\\
=  \sum\limits_{\alpha \in \Lpo }\E\left[\left|\E\left[\mathbf{1}_{Z_{\alpha}>a_n}|\sigma\left(Z_{\beta}, \beta \in (\Lpo \cup \Lpe)\backslash B_{\alpha}\right)\right]-\E\left[\mathbf{1}_{Z_{\alpha}>a_n}\right]\right|\right]\\
+  \sum\limits_{\alpha \in  \Lpe}\E\left[\left|\E\left[\mathbf{1}_{Z_{\alpha}>a_n}|\sigma\left(Z_{\beta}, \beta \in (\Lpo \cup \Lpe)\backslash B_{\alpha}\right)\right]-\E\left[\mathbf{1}_{Z_{\alpha}>a_n}\right]\right|\right]\\
\end{multline}

The first term of the RHS above is $0$ from the choice of $B_{\alpha}$.
Hence

\begin{eqnarray}
b_{3,n}
& = & \sum\limits_{\alpha \in  \Lpe}\E\left[\left|\E\left[\mathbf{1}_{Z_{\alpha}>a_n}|\sigma\left(Z_{\beta}, \beta \in (\Lpo \cup \Lpe)\backslash B_{\alpha}\right)\right]-\E\left[\mathbf{1}_{Z_{\alpha}>a_n}\right]\right|\right]\\
& \leqslant & \sum\limits_{\alpha \in  \Lpe}\E\left[\E\left[\mathbf{1}_{Z_{\alpha}>a_n}|\sigma\left(Z_{\beta}, \beta \in (\Lpo \cup \Lpe)\backslash B_{\alpha}\right)\right]\right]+\E\left[\E\left[\mathbf{1}_{Z_{\alpha}>a_n}\right]\right]\\
& \leqslant & 2\left|\Lpe\right|\fPk\\
\end{eqnarray}
According to the hypotheses in theorem \ref{hyp5}, we have 
$$\liminfini\left[\left|\Lpe\right|\fPk\right]=0$$ 
Finally, we obtain :

$$\liminfini\left[b_{3,n}\right]=0$$

\begin{rem}Gathering remark \ref{rem1} to \ref{rem5} and noticing that on $\gamma\in ]-2+\sqrt 2,2-\sqrt 2[$ we have $\frac{1}{3}(c_{\gamma}-1)<\frac{1}{6}c_{\gamma}$ we get the final assumptions of theorem \ref{hyp5} for $\nu$.
\end{rem}

\hfill$\square$
{\subsection{Proof of Theorem  \ref{main_result}}
We showed in Section \ref{chenStenVntauprime} that

	\begin{eqnarray}
	\fP\left(V_{n,\tau}'\leqslant a_n\right)=\exp\left(-\lambda_n\right)+o\left(1\right) \text{  as  } n \rightarrow +\infty
	\end{eqnarray}
	Thanks to  lemma \ref{equiVtau}, we have, for $n$ large enough :
	\begin{eqnarray}
	\fP\left(V_{n,\tau}\leqslant a_n\right)=\exp\left(-\lambda_n\right)+o\left(1\right) \text{  as  } n \rightarrow +\infty
	\end{eqnarray}
From the expressions of $a_n$ and $\lambda_n$, this leads us to the asymptotic behaviour :
	\begin{eqnarray}
	\frac{1}{n}V_{n,\tau}^2-4\log\left(p_n\right)+\log\log\left(p_n\right) \underset{n\rightarrow +\infty}{\overset{\mathcal{L}}{\longrightarrow}} Z
	\end{eqnarray}
	where $Z$ has the Gumbel cdf defined in theorem \ref{main_result}. Then, we can write :
	\begin{eqnarray}
	\frac{1}{n\log\left(p_n\right)}V_{n,\tau}^2 -4 \underset{n\rightarrow +\infty}{\overset{\fP}{\longrightarrow}} 0
	\end{eqnarray}
	Then, from Proposition \ref{convLnvn} we have \eqref{mainlimit}.

\hfill$\square$
}
\section{Proofs of technical results}\label{tech}
\subsection{Proof of Proposition \ref{convLnvn}}
We first recall here the basic definition of tightness :
\begin{defn}\label{deftight}
	Let $(u_n)$ be a real sequence. We say that $(u_n)$ is a tight sequence if :
	\begin{equation}
	\forall \epsilon >0, \exists K >0 \text{,  } \underset{n\geq 1}{\sup} \left(\fP\left( |u_n| \geq K \right) \right)< \epsilon
	\end{equation}
\end{defn}
\par\noindent
For the proof of \cref{convLnvn} we use the following lemma which is proved in \cite{Jiang04} (see Lemma 2.2):

\begin{lem}\label{NormeJiang}
	Let $\tau$ be an integer. We define $|||A|||=\underset{1\leq i < j \leq p, |i-j| > \tau}{\max}|A_{ij}|$ for a\\ $(p\times p)$-matrix.  Let $\mathbf{X}$ be a random $(n,p)$-matrix where $(X^1, X^2, \dots, X^p)$ are the $p$ columns in $\R^n$ and $R_n$ the empirical correlation matrix of $\mathbf{X}$. Let's define, for any $k \in [\![1,p]\!]$ :
	\begin{itemize}
		\item $h_k=\frac{1}{\sqrt{n}}\parallel X^k-\overline{X^k}\mathbf{1}_n\parallel$ with $\overline{X^k}=\frac{1}{n}\sum\limits_{i=1}^{n}X^k_i$
		\item $c_{n,1}=\underset{1\leqslant k \leqslant p}{\max}|h_k-1|$
		\item $c_{n,2}  =\underset{1\leqslant k \leqslant p}{\min} h_k$
		\item $c_{n,3}  =\underset{1\leqslant k \leqslant p}{\max}|\overline{X^k}|$
		\item $V_{n,\tau}=\underset{1\leqslant k < j \leqslant p,|k-j|\geqslant \tau}{\max}|{}^tX^kX^j|=\underset{1\leqslant k < j \leqslant p,|k-j|\geqslant \tau}{\max}|\sum\limits_{i=1}^{n}X^k_iX^j_i|$
	\end{itemize}
	Then, 
	\begin{equation} |||nR_n -{}^t\mathbf{X}\mathbf{X}||| \leqslant \frac{c_{n,1}^2+2c_{n,1}}{c_{n,2}  ^2}V_{n,\tau}+n\left(\frac{c_{n,3}  }{c_{n,2}  }\right)^2
	\end{equation}
\end{lem}
\medskip\pn
We denote here by
	$\Delta_n := \left|nL_{n,\tau}-V_{n,\tau}\right|$, for any $n\geqslant 1$. We have :
	\begin{equation}
	\left|n^2L_{n,\tau}^2-V_{n,\tau}^2\right|=\left|nL_{n,\tau}-V_{n,\tau}\right|.\left|nL_{n,\tau}+V_{n,\tau}\right| \leqslant \Delta_n . \left(\Delta_n +2V_{n,\tau}\right)
	\end{equation}
	Now, we can notice :\\
	\begin{equation}
	\Delta_n \leqslant |||nR_n-{}^t\mathbf{X}\mathbf{X}||| {\leqslant} \frac{c_{n,1}^2+2c_{n,1}}{c_{n,2}  ^2}V_{n,\tau}+n\left(\frac{c_{n,3}  }{c_{n,2}  }\right)^2
	\end{equation}
	
	where the upper bound above refers to lemma \ref{NormeJiang}. Using definition \ref{deftight}, we see that \\ $\left(c_{n,1}\sqrt{\frac{n}{\log(p_n)}}\right)_{n\geqslant 1}:=\left(c_{n,1}'\right)_{n\geqslant 1}$ and $\left(c_{n,3}  \sqrt{\frac{n}{\log(p_n)}}\right)_{n\geqslant 1}$:=$\left(c_{n,3}  '\right)_{n\geqslant 1}$ are both tight sequences. In that way, both sequences $c_{1,n}=c_{n,1}'\sqrt{\frac{\log(p_n)}{n}}$ and $c_{4,n}=c_{4,n}'\sqrt{\frac{\log(p)}{n}}$ are tight too (from Hyp 1 in theorem \ref{hyp2}, $\frac{\log(p_n)}{n}\rightarrow 0$ when $n\rightarrow +\infty$).\\
	So, 
	\begin{eqnarray*}
		\Delta_n & \leqslant & \frac{\frac{\log(p_n)}{n}c^{'2}_{n,1}+2\sqrt{\frac{\log(p_n)}{n}}c_{n,1}'}{c_{n,2}  ^2}V_{n,\tau}+n\frac{\log(p_n)}{n}\left(\frac{c_{n,3}  '}{c_{n,2}  }\right)^2\\
		\frac{\Delta_n}{\log(p_n)} & \leqslant & \frac{\frac{1}{n}c^{'2}_{n,1}+2\sqrt{\frac{1}{n}}c_{n,1}'}{c_{n,2}  ^2}V_{n,\tau}+\left(\frac{c_{n,3}  '}{c_{n,2}  }\right)^2\\
		\frac{\Delta_n}{\log(p_n)} & \leqslant & \frac{V_{n,\tau}}{\sqrt{n\log(p_n)}}\frac{c_{n,1}^{'2}\sqrt{\frac{\log(p_n)}{n}}+2c_{n,1}'}{c_{n,2}  ^2}+\left(\frac{c_{n,3}  '}{c_{n,2}  }\right)^2\\
	\end{eqnarray*}
	With this inequality, we notice that the sequence $(\Delta_n')_{n\geqslant 1}:=\left(\frac{\Delta_n}{\log(p_n)}\right)_{n\geqslant 1}$ is tight. 
 Finally, with tightness of $\left(\Delta_n'\right)_{n\geqslant 1}$, we have from assumptions in theorem \ref{hyp1} :
	\begin{eqnarray*}
		\frac{\left|n^2L_{n,\tau}^2-V_{n,\tau}^2\right|}{n} & \leqslant & \frac{\Delta_n}{n}\left(\Delta_n+2V_{n,\tau}\right)=\frac{\Delta_n'\log(p_n)}{n}\left(\Delta_n'\log(p_n)+2V_{n,\tau}\right)\\
		& \leqslant & 2\sqrt{\frac{\left(\log(p_n)\right)^3}{n}}\left(\Delta_n'\sqrt{\frac{\log(p_n)}{n}}+V_{n,\tau}\right)\underset{n\rightarrow +\infty}{\overset{\fP}{\longrightarrow}}0.\\
	\end{eqnarray*}

\subsection{Proof of Lemma \ref{probI0}}
For the proof of this Lemma, we need the following technical result which is presented in \cite{CaiJiang} (see Lemma 6.8) and proved in the supplementary paper.
\begin{lem}\label{lemme68}
	We consider the following hypotheses :
	\begin{enumerate}
		\item $\xi_1, \dots, \xi_n$ i.i.d random variables such that $\E[\xi_1]=0$ and $\E[\xi_1^2]=1$.
		\item $\exists t_0>0$, $\exists\alpha \in ]0,1]$ such that $\E\left[e^{t_0|\xi_1|^{\alpha}}\right]<+\infty$.
		\item $(p_n)_{n\in \N^*}$ such that $p_n\underset{n\rightarrow +\infty}{\longrightarrow}+\infty$ and $\log(p_n)=o\left(n^{\frac{\alpha}{2+\alpha}}\right)$ as ${n\rightarrow +\infty}$
		\item $(y_n)_{n \geqslant 1}$ such that $y_n \underset{n\rightarrow +\infty}{\longrightarrow}y >0$
	\end{enumerate}
	Then, 
	\begin{equation} \fP\left(\frac{1}{\sqrt{n\log(p_n)}}\sum\limits_{k=1}^{n}\xi_k \geqslant y_n\right)\underset{n\rightarrow +\infty}{\sim}\frac{1}{y\sqrt{2\pi}}p_n^{-\frac{1}{2}y_n^2}\sqrt{\log(p_n)}^{-1}
	\end{equation}
\end{lem}

Let us check all the hypotheses of this lemma above. First we write :
	\begin{equation} 
	\fP\left(\left| {}^tX^1X^{\tau+ K +2}\right|>a_n\right)=\fP(\left(\left| \sum\limits_{i=1}^{n}X^1_iX^{\tau+ K +2}_i\right|>a_n\right)
	\end{equation} 
	Now, if we define $\xi_i=X^1_iX^{\tau+ K +2}_i$, we have :
	\begin{enumerate}
		\item $\displaystyle \E[\xi_i] \overset{\indep}{=} \E[X^1_i]\E[X^{\tau+ K +2}_i]= 0 \times 0 =0$ where the independence come from the sample.
		\item $\E[\xi_i^2] \overset{\indep}{=} \E[\left(X^1_i\right)^2]\E[\left(X^{\tau+ K +2}_i\right)^2]=1 \times 1 =1$
		\item For $t_0=\frac{1}{2}$ and $\alpha=1$, we have : $$\E[e^{t_0|\xi_1|^{\alpha}}]=\E[e^{\frac{|X_1^1X_1^{\tau+K+2}|}{2}}]\leqslant \E[e^{\frac{1}{2}\left(X_1^1\right)^2}]\E[e^{\frac{1}{2}\left(X_1^{\tau +K+2}\right)^2}]<+\infty$$
		\item We have $w_n := \frac{a_n}{\sqrt{n\log(p)}} \ninf \sqrt{4}=2 >0$
		\item According to the hypothesis 1 from theorem \ref{main_result} : $\log(p_n)=o(n^{\frac{1}{3}})$ as $n \ninf +\infty$
	\end{enumerate}
	So we have all hypothesis needed to apply the lemma \ref{lemme68}, and then :\\
	\begin{multline}
		\fP\left(\left|{}^tX^1X^{\tau + K + 2}\right|>a_n\right)  =  \fP\left(\frac{1}{\sqrt{n\log(p_n)}}\left|{}^tX^1X^{\tau + K + 2}\right|>\frac{a_n}{\sqrt{n\log(p_n)}}\right)\\
		=  \fP\left(\frac{1}{\sqrt{n\log(p_n)}}\left|\sum\limits_{i=1}^{n}X^1_iX^{\tau + K + 2}_i\right|>\frac{a_n}{\sqrt{n\log(p_n)}}\right)\\
		=  \fP\left(\frac{1}{\sqrt{n\log(p_n)}}\sum\limits_{i=1}^{n}X^1_iX^{\tau + K +2}_i>\frac{a_n}{\sqrt{n\log(p_n)}}\right)
		 +  \fP\left(\frac{1}{\sqrt{n\log(p_n)}}\sum\limits_{i=1}^{n}X^1_iX^{\tau + K + 2}_i < -\frac{a_n}{\sqrt{n\log(p_n)}}\right)\\
	 =  \fP\left(\frac{1}{\sqrt{n\log(p_n)}}\sum\limits_{i=1}^{n}X^1_iX^{\tau + K +2}_i>w_n\right)
		 +  \fP\left(-\frac{1}{\sqrt{n\log(p_n)}}\sum\limits_{i=1}^{n}X^1_iX^{\tau + K + 2}_i > w_n\right)
	\end{multline}
	From lemma \ref{lemme68} 
	\begin{eqnarray*}
	\fP\left(\left|{}^tX^1X^{\tau + K + 2}\right|>a_n\right) &=& \frac{1}{2\sqrt{2\pi}}p_n^{-\frac{1}{2}w_n^2}\frac{1}{\sqrt{\log(p_n)}}(1+o(1))+\frac{1}{2\sqrt{2\pi}}p_n^{-\frac{1}{2}w_n^2}\frac{1}{\sqrt{\log(p_n)}}(1+o(1))\\
	&=&  \frac{1}{\sqrt{2\pi}}p_n^{-\frac{1}{2}w_n^2}\frac{1}{\sqrt{\log(p_n)}}(1+o(1))\\
	&=&  \frac{1}{2\pi\log(p_n)}e^{-\frac{1}{2}\frac{a_n^2}{n\log(p_n)}-\frac{1}{2}\log\log(p_n)}(1+o(1)) =  \frac{1}{p_n^2\sqrt{2\pi}}e^{-\frac{1}{2}y}(1+o(1))
	\end{eqnarray*}	

\subsection{Proof of Lemma lemma \ref{probIk}}

We remind that $\fPk := \fP\left(\left| {}^tX^1X^{\tau+1}\right|>a_n\right)$. We will apply once again lemma \ref{lemme68}, with new quantities $\xi$ and $\omega$:

\begin{itemize}
	\item[$\bullet$] $\displaystyle \fPk^+:=\fP\left({}^tX^1X^{\tau+1}>a_n\right)$
	\item[$\bullet$] $\displaystyle \fPk^-:=\fP\left({}^tX^1X^{\tau+1}<-a_n\right)$
	\item[$\bullet$] $\displaystyle \xi_k:=X_k^1X_k^{\tau+1}$
	\item[$\bullet$] $\displaystyle w_k:=\frac{1}{\sqrt{1+\eps_n^2\left(\xi_k-\eps_n\right)}}$.\\
\end{itemize}

Notice that $\left(\xi_k\right)_{k\geqslant 1}$ are independent due to the independence between each line of $\mathbf{X}_n$. First we compute $\E\left[\xi_k\right]=\eps_n$ and $\text{var}\left(\xi_k\right)=1+\eps_n^2$. So, $\E\left[w_k\right]=0$ and $\text{var}\left(w_k\right)=1$.We will apply the lemma \ref{lemme68} with $w_k$. Then,

\begin{equation}
\fPk^+  =  \fP\left(\frac{1}{\sqrt{n\log(p_n)}}\sum\limits_{k=1}^{n}w_k > \underbrace{ \frac{a_n-n\eps_n}{\sqrt{(1+\eps_n^2)n\log(p_n)}}}_{:=z_n}\right)
\end{equation}

From hypotheses of theorem \ref{main_result}, we have $\liminfini\left[z_n\right]:=z=2-\gamma>0$. Then,

\begin{eqnarray*}
	\fPk^+ & \sim & \frac{1}{z\sqrt{2\pi}}p_n^{-\frac{1}{2}z_n^2}\sqrt{\log(p_n)}^{-1}\\
	& \sim & \frac{1}{z\sqrt{2\pi}}\exp\left[-\frac{1}{2}z_n^2\log\left(p_n\right)-\frac{1}{2}\log\log\left(p_n\right)\right]\\
		& \sim & \frac{1}{z\sqrt{2\pi}}\exp\left[ \frac{-2}{1+\eps_n^2}\log\left(p_n\right)\left(1+\frac{\eps_n^2}{4}\frac{\log\log\left(p_n\right)}{\log\left(p_n\right)}+\frac{y}{4}\frac{1}{\log\left(p_n\right)}+\frac{n\eps_n^2}{4\log\left(p_n\right)}-\frac{\eps_n}{2}\frac{a_n}{\log\left(p_n\right)}\right) \right]\\
\end{eqnarray*}

With our hypotheses on $\eps_n$, we have :

\begin{itemize}
	\item[$\bullet$] $\displaystyle \frac{-2}{1+\eps_n^2}\log\left(p_n\right) \underset{n\rightarrow +\infty}{\longrightarrow} - \infty$
	\item[$\bullet$] $\displaystyle\frac{\eps_n^2}{4}\frac{\log\log\left(p_n\right)}{\log\left(p_n\right)} \underset{n\rightarrow +\infty}{\longrightarrow} 0$
	\item[$\bullet$] $\displaystyle\frac{y}{4}\frac{1}{\log\left(p_n\right)} \underset{n\rightarrow +\infty}{\longrightarrow} 0 $
	\item[$\bullet$] $\displaystyle\frac{n\eps_n^2}{4\log\left(p_n\right)} \underset{n\rightarrow +\infty}{\longrightarrow} \frac{1}{4}\gamma^2 $ from $\eps_n \sim \gamma\sqrt{\frac{\log(p_n)}{n}}$.
	\item[$\bullet$] $\displaystyle\frac{\eps_n}{2}\frac{a_n}{\log\left(p_n\right)} \underset{n\rightarrow +\infty}{\longrightarrow} \gamma $ from $a_n \sim 2\sqrt{n\log(p_n)}$
\end{itemize}

\begin{equation}
p_n^a\fPk^+ \sim \frac{1}{z\sqrt{2\pi}}\exp\left[
\left(a-2-\frac{1}{2}\gamma^2+2\gamma+o\left(1\right)\right)\log\left(p_n\right)
\right]\\
\end{equation}

Finally, for $\gamma \in ]-2,2[$ :

\begin{equation}\label{fpkplusnul}
\liminfini\left[p_n^a\fPk^+\right]=0 \Leftrightarrow a < 2+\frac{1}{2}\gamma^2-2\gamma
\end{equation} 
Analogously, we have : 

\begin{eqnarray}
\fPk^-  & = &  \fP\left(\frac{1}{\sqrt{n\log(p_n)}}\sum\limits_{k=1}^{n}w_k < - \frac{a_n+n\eps_n}{\sqrt{(1+\eps_n^2)n\log(p_n)}}\right)\\
& = &  \fP\left(\frac{-1}{\sqrt{n\log(p_n)}}\sum\limits_{k=1}^{n}w_k > \underbrace{ \frac{a_n+n\eps_n}{\sqrt{(1+\eps_n^2)n\log(p_n)}}}_{:=\tilde{z_n}}\right)
\end{eqnarray}

Thanks to lemma \ref{lemme68} and because $\liminfini\left[\tilde{z_n}\right]:=\tilde{z}=2+\gamma>0$, we have :

\begin{equation}
\fPk^- \sim \frac{1}{\tilde{z}\sqrt{2\pi}}\exp\left[ \frac{-2}{1+\eps_n^2}\log\left(p_n\right)\left(1+\frac{\eps_n^2}{4}\frac{\log\log\left(p_n\right)}{\log\left(p_n\right)}+\frac{y}{4}\frac{1}{\log\left(p_n\right)}+\frac{n\eps_n^2}{4\log\left(p_n\right)}+\frac{\eps_n}{2}\frac{a_n}{\log\left(p_n\right)}\right) \right]
\end{equation}

\begin{equation}
p_n^b\fPk^- \sim \frac{1}{\tilde{z}\sqrt{2\pi}}\exp\left[
\left(b-2-\frac{1}{2}\gamma^2-2\gamma+o\left(1\right)\right)\log\left(p_n\right)
\right]
\end{equation}

And finally, for $\gamma \in ]-2,2[$ :

\begin{equation}\label{fpkmoinsnul}
\liminfini\left[p_n^b\fPk^-\right]=0 \Leftrightarrow b < 2+\frac{1}{2}\gamma^2+2\gamma
\end{equation} 

To conclude, observing that $$\min\left(2+\frac{1}{2}\gamma^2+2\gamma,2+\frac{1}{2}\gamma^2-2\gamma\right)=\frac{1}{2}\gamma^2-2|\gamma|+2:=c_{\gamma},$$
combining \cref{fpkmoinsnul} and \cref{fpkplusnul}, we obtain, for all $d \in [0;c_{\gamma}[$ and as $n \rightarrow +\infty$ :

\begin{equation} 
\fPk:= \fP\left(\left| {}^tX^1X^{\tau+1}\right|>a_n\right)=o\left(p_n^{-d}\right)\,.
\end{equation}
\medskip\par
Acknowledgments: the authors wish to thank Laurent Delsol for many fruitful discussions on the model.
\bibliographystyle{alpha} 

\bibliography{bib_rapport_5}

\end{document}